\documentclass[12pt,a4paper,reqno]{amsart}
\usepackage{amsmath,amssymb,amsthm,amscd,mathrsfs} 
\usepackage[T1]{fontenc}
\usepackage[latin1]{inputenc}
\usepackage[english]{babel}
\usepackage[mathscr]{eucal}
\usepackage{graphicx}
\usepackage{float}
\usepackage{tikz}
\usetikzlibrary{positioning}
\usetikzlibrary{shapes.geometric}
\usetikzlibrary{fit}
\usetikzlibrary{patterns}
\usepackage{caption}
\usepackage{tikz-cd}
\usepackage{cases}
\usepackage{hyperref}
\usepackage{chngcntr}

\usetikzlibrary{decorations.text}
\font\sc=rsfs10 at 12pt

\input{comment.sty}
\includecomment{FZ}
\includecomment{MM}
\counterwithin{equation}{section}
\newcommand{\om}{\omega}
\renewcommand{\a}{\alpha}

\renewcommand{\th}{\theta}

\renewcommand{\k}{\kappa}

\renewcommand{\l}{\lambda}

\newcommand{\n}{\nu}

\newcommand{\ph}{\phi}
\def\ph{\phi}
\def\md#1{\ \mbox{\rm(mod }{#1})}
\def\nph#1{N_{\ph}(#1)}
\def\npp#1{N_{\ph}^+(#1)}
\def\ol{\overline}
\renewcommand{\t}{\tau}

\def\ph{\phi}

\newcommand{\Q}{{\mathbb Q}}
\newcommand{\Z}{{\mathbb Z}}

\newcommand{\F}{{\mathbb F}}

\def\md#1{\ \mbox{\rm(mod }{#1})}
\def\nph#1{N_{\ph}(#1)}
\def\npp#1{N_{\ph}^+(#1)}
\def\nppsi#1{N_{\psi}^+(#1)}

\newcommand{\aF}{\mathfrak a}
\newcommand{\bF}{\mathfrak b}

\newcommand{\pF}{\mathfrak p}

\newtheorem{theorem}{Theorem}[section]
\newtheorem{proposition}[theorem]{Proposition}
\newtheorem{lemma}[theorem]{Lemma}
\newtheorem{corollary}[theorem]{Corollary}

\theoremstyle{definition}

\theoremstyle{remark}
\newtheorem{remark}[theorem]{Remark}

\newtheorem{example}[theorem]{Example}
\begin{document}
\title[]{On monogenity of certain  number fields defined by trinomials}
\textcolor[rgb]{1.00,0.00,0.00}{}
\author{Hamid Ben Yakkou and Lhoussain El Fadil}\textcolor[rgb]{1.00,0.00,0.00}{}
\address{Faculty of Sciences Dhar El Mahraz, P.O. Box  1796 Atlas-Fes , Sidi mohamed ben Abdellah University,  Morocco}\email{beyakouhamid@gmail.com\\lhouelfadil2@gmail.com}
\keywords{ Power integral bases,  theorem of Ore, prime ideal factorization, common index divisor}
 \subjclass[2010]{11R04,11R16, 11R21}
\date{\today}
\maketitle
\vspace{0.3cm}
\begin{abstract}
  Let $K=\Q(\theta)$ be a number field generated  by a complex  root  $\th$ of a monic irreducible trinomial  $F(x) = x^n+ax+b \in \Z[x]$. There is an extensive literature of monogenity of number fields defined by trinomials, Ga\'al studied the multi-monogenity of sextic number fields defined by trinomials. Jhorar and Khanduja  studied the integral closedness  of $\Z[\th]$. But if  $ \Z[\th]$ is not integrally closed, then Jhorar and Khanduja's results cannot answer on the monogenity of $K$. In this paper, based on Newton polygon techniques, we deal with the problem of  monogenity of $K$. More precisely, when  $\Z_K \neq \Z[\th]$,  we give sufficient conditions on $n$, $a$ and $b$ for $K$ to be not monogenic. For $n\in  \{5, 6, 3^r,  2^k\cdot 3^r, 2^s\cdot 3^k+1\}$, we give explicitly some  infinite  families of these number fields that are not monogenic. Finally, we illustrate our results by some computational examples.        
\end{abstract}
\maketitle

\section{Introduction}
Let $K=\Q(\th)$ be a number field  generated by a complex  root   $\th$ of a monic irreducible  polynomial $f(x)$ of degree $n$ over  $\Q$ and   $\Z_K$   its  ring of  integers. Denote by $\widehat{\Z}_K$ the set of all primitive integral  elements of $K$. The field $K$ is called monogenic if $\Z_K$  has a power integral basis. Namely $\Z_K=\Z[\eta]$ for some $\eta \in \widehat{\Z}_K$.  In this case  $(1,\eta,\ldots,{\eta}^{n-1})$ is a power integral basis of $K$. Thus, if $\Z[\eta]$ is integrally closed for some $\eta \in \widehat{\Z}_K $, then $K$ is monogenic. If $\Z_K$ has no power integral basis, we say that  $K$ is not monogenic. Monogenity is a classical problem of algebraic number theory, going back to Dedekind, Hasse and Hensel \cite{ G19, Ha, Hedisc,   MNS}. There is an extensive  computational results  in the literature of testing the monogenity of number fields and constructing power integral basis, and it was treated by different approachs. Ga\'{a}l, Gy\H{o}ry, Pohst, and Peth\"{o}  (see \cite{Gacta2, G19, Gacta1,  GRN, GR17, PG}) with their  research teams  based on arithmetic index form equations, they studied the monogenity of several algebraic number fields. In\cite{Gacta1}, Ga\'{a}l and Gy\H{o}ry described an algorithm to solve  index form equations in quintic fields  and they computed all generators of power integral bases in totally real quintic field with Galois group $S_5$. In  \cite{Gacta2}, Bilu, Ga\'{a}l and Gy\H{o}ry   studied  the monogenity of totally real sextic fields with Galois group $S_6$. In \cite{GR17}, Ga\'{a}l and Remete answered   completely  to the problem of  monogenity of pure number fields $K= \Q(\sqrt[n]{m})$, where $m \neq  \pm 1$ is a square free rational integer and $3 \le n \le 9$. In \cite{GRN},  Ga\'{a}l and Remete showed that if $m \equiv  2\, \mbox{or} \,3 \md 4$ is square free rational integer, then the octic field $K= \Q(i, \sqrt[4]{m} )$ is not monogenic.
Nakahara's research team   based on the existence of power relative integral bases of some special  sub-fields, they studied the monogenity of some pure number fields (see \cite{ANHN, ANHN6, HN8, MNS}). S. Ahmad, T. Nakahara and S. M. Husnine \cite{ANHN} proved that if $m \equiv 2, 3 \md 4$ and $m \not \equiv  \pm 1 \md 9 $ is a square free rational integer then the sextic pure field $K = \Q(\sqrt[6]{m})$ is monogenic. On the other hand  \cite{ANHN6} if $m \equiv 1 \md 4$ and $m \not \equiv  \pm 1 \md 9 $  then the sextic pure field $K = \Q(\sqrt[6]{m})$ is not mongenic. They also studied in \cite{HN8} the monogenity of certain   pure octic fields.  In \cite{Gras}, Gras proved that except the real maximal sub-fields of the cyclotomic fields, all cyclic fields of prime degree $l \ge 5$  are not monogenic. In \cite{Smtacta},  Smith studied the monogenity of radical extensions and  he  gave  sufficient conditions for a  Kummer extension $\Q(\xi_n, \sqrt[n]{\a})$  to be not monogenic.  He also studied  in \cite{Smit acta2}  the monogenity of two families  $S_4$ quartic number fields.  In \cite{E07}, El Fadil gave conditions for the existence of a generator of power integral basis of pure cubic fields in terms of index form equation. Based on prime ideal factorization, El Fadil showed in \cite{El6} that for  a square free rational integer $m \neq \pm 1$ if $m\equiv 1 \md 4$ or $m \not \equiv 1 \md 9$, then the pure sextic field $\Q(\sqrt[6]{m})$ is not monogenic. He also studied in \cite{Eljnt}  the monogenity of $\Q(\sqrt[6]{m})$ where  $m$ is not necessarily  square free. In \cite{EL24}, El Fadil studied the monogenity of pure number fields of degree $24$. He also  studied in \cite{Facta} the monogenity of pure number fields of degree $2\cdot 3^k$. 
In cite{Ga21}, Ga\'al studied the multi-monogenity of sextic number fields defined by trinomials. 
In \cite{BFC, BF, BFN, BK}, Ben Yakkou et al.  considered the problem of monogenity  in certain pure number fields  with large degrees, namely $2^r\cdot5^k$, $p^r$,   $2^r\cdot3^k$, and $3^r$, with $p$ is a rational prime integer, $r$ and $k$ are two positive rational integers. 
In this paper, based on Newton polygon techniques applied on  prime ideal factorization, we study the  monogenity  of number  fields  $K= \Q(\th)$ generated by a complex  root  $\th $ of a monic irreducible   trinomial of the type $x^n +ax+b$  when $ \Z_K \neq \Z[\th]$.  
Recall that the problem of integral-closedness  of $\Z[\th]$ has been previously studied  in \cite{SK} and refined in \cite{Smt} by Ibarra et al with certain computation of densities.  But if  $ \Z[\th]$ is not integrally closed, then their results cannot answer on the monogenity of $K$.

\section{Main Results}
 Let $p$ be  rational prime integer. Throughout this paper, $\F_p$ denotes the finite field of $p$ elements. For $t \in \Z$, $\n_p(t)$ stands for the $p$-adic valuation of $t$,  and  $t_p $ for the image of  $ \frac{t}{p^{\nu_p(t)}}$ under the canonical projection from $\Z$ onto $\F_p$. For two positive rational integers $m$ and $s$,  we shall denote by $N_p(m)$ the number of  monic irreducible polynomials of degree $m$ in $\F_p[x]$ and
$N_p(m,s,t)$ the number of monic irreducible factors of degree $m$ of the polynomial $x^s+\overline{t}$ in $\F_p[x]$. It is known from \cite{Greenfield} that  the discriminant of the trinomial $F(x) = x^n +ax +b $ is 
\begin{eqnarray}\label{disctrinom}
\Delta(F)= (-1)^{\frac{n(n-1)}{2}} (n^n b^{n-1} + (1-n)^{n-1}a^n).
\end{eqnarray}It follows by (\ref{indexdiscrininant}) and (\ref{i(K)}) that,  if a rational  prime integer  $q$ divides $i(K)$, then  $q^2 \mid \Delta(F)$. Without loss of generality we assume that for every rational prime  integer$q$,  $\n_q(a)  <n-1$ or  $ \n_q(b) < n $. We shall make this assumption for finding some  suitable  conditions of Theorems \ref{d51} and \ref{d61}. We note also that if a rational prime $p$ satisfies one of the conditions of Theorems \ref{dp^r}, \ref{dn1} and \ref{dn2}, then $p$ divides  $(\Z_K : \Z[\theta])$ and so $ \Z[\theta]$ is not the ring of integers of $K$.  In the remainder of this section, $K=\Q(\th)$ is a number field generated  by a complex  root  $\th$ of a monic irreducible trinomial of the type  $F(x) = x^n + ax +b \in \Z[x]$ and  $\Z_K$  its ring of   integers. We start in Theorem  \ref{mono} by giving an example of number field generated by a complex  root  $\th$ of an   irreducible trinomial such that $\Z[\th]$ is not integrally closed, but $K$ is monogenic. Remark that in this case, the results given in \cite{SK} and \cite{Smt} can not give an answer to the monogenity of $K$. 
\begin{theorem}\label{mono}
		Let $p$ a rational prime integer, $F(x) = x^{p^r} + p^v  a x + p^u b \in \Z[x] $ such that  $p \nmid ab$, $v \ge u \ge 2$ and   $\n_q((1-p^r)^{p^r-1}\cdot (p^v a)^{p^r}+(p^r)^{p^r}\cdot (p^v b)^{p^r-1} ) \le 1$ for every rational prime integer $q\neq p$. If $\gcd(u,p)=1$, then  F(x) is irreducible over $\Q$. Let $K$ be the number field genereted by a complex root $\th$ of $F(x)$, then $\Z[\th]$ is not integrally closed, $K$ is monogenic, and  $\eta =\frac{\theta^x}{p^y}$   generates a power integral basis of $\Z_K$, where $(x,y)\in \Z^2$ is the unique solution of  $xu-yp^r = 1$ with $0 \le y < u$.
\end{theorem}
\begin{theorem}\label{dp^r}
	Let $p$ be an odd rational prime integer, $F(x)= x^{p^r} +ax +b \in \Z[x]$. If  $a \equiv 0 \md{p^{p+1}}$, $b^{p-1}\equiv 1 \md{p^{p+1}}$, and $r \ge p$, then $K$ is not monogenic.
\end{theorem}
\begin{remark}
 Theorem \ref{dp^r}, implies \cite[Theorem 2.4]{BF}, where $a=0$ is previously studied. 
\end{remark}
\begin{corollary}\label{3^r} For $F(x)= x^{3^r}+ax+b \in \Z[x]$.
	If $a \equiv 0 \md{81}$, $b\equiv \pm 1 \md{81}$, and $r \ge 3$. Then $K$ is not monogenic.
\end{corollary}

\begin{theorem}\label{dn1}
	For $F(x) = x^n+ ax+b \in \Z[x]$, let $p$ be an odd rational prime integer such that $p \mid a$, $p \mid n$ and $p \nmid b$. Set $n = s\cdot p^r$, where $p \nmid s$. Let $\mu = \n_p(a)$, $\nu =\n_p(b^{p-1}-1)$, $\delta = \min (\mu, \nu)$, and $\omega = \min(\delta,r+1)$. If one of the following conditions holds:
\begin{enumerate} 
	\item  $\delta \neq r+1$ and $\omega > \frac{N_p(m)}{N_p(m,s,b)}$ for some $m > 1$.
	\item  $\frac{p}{N_p(1,s,b)}<\mu< \min(\nu,r+1) $,
	\item  $\frac{p}{N_p(1,s,b)}<\nu <\min(\mu,r+1)$,
	\item $\frac{p}{N_p(1,s,b)}<r+1 < \delta$,
\end{enumerate}	
	then $K$ is not monogenic.
\end{theorem}
The following corollary gives certain  infinite families of non-monogenic  number fields defined by irreducible trinomials of degree $2^k \cdot 3^r$, where $k$ and $r$ are two positive rational integers. 
\begin{corollary}\label{corn11} Let $k$, $r$ be two positive rational integer, and $F(x) = x^{2^k \cdot 3^r} + ax +b$. If one of the following conditions holds:
	\begin{enumerate}
		
		\item $k  \ge 1$, $r = 3$, $a \equiv 0 \md{243} $ and $b \equiv -1 \md{243}$,
		\item $k \in \{1,2\} $, $r \ge 4$ and   $(\bar{a},\bar{b}) \in  \{\overline{81}, \overline{162}\}\times  \{ \overline{80}, \overline{161}, \overline{242} \} \cup \{ \overline{0}\}\times \{\overline{80}, \overline{161} \}$ in $(\Z/{243\Z})^2$, 
		\item $k \geq 1 $,  $r=1$, $a \equiv 0 \md {27}$ and $b \equiv -1 \md {27}$,
		\item $k\ge 1$, $r \ge 2$ and $(\bar{a},\bar{b})\in  \{\overline{9}, \overline{18}\}\times \{\overline{26} \} \cup \{\overline{0} \} \times \{ \overline{8}, \overline{17} \}$ in $(\Z/{27\Z})^2$,
		\item $k \ge 1$,  $r \ge 3$ and 	$(\bar{a}, \bar{b}) \in  \{\overline{0}\}\times  \{\overline{26}, \overline{53}\}\cup \{\overline{27}, \overline{54}\}\times \{\overline{26}, \overline{53},\overline{80}\}$ in $(\Z/{81\Z})^2$,
		\item  $k \ge 1$, $r=2$, $ a\equiv 0 \md {81}$ and $b\equiv  -1 \md{81}$,
		\item $k=1$, $r=3$, $a \equiv 0 \md{243}$ and $b\equiv 1 \md{243}$,
		\item $k=1$, $r\ge 4$ and $(\bar{a}, \bar{b})\in \{\overline{81}, \overline{162}\}\times 
	\{\overline{1}, \overline{82}, \overline{163}\}\cup \{\overline{0}\}\times \{\overline{82}, \overline{163}\}$ in $(\Z/{243\Z})^2$. 
		\item $k=2$, $r=1$, $a\equiv 0 \md{27}$ and $b \equiv 1 \md{27}$,
		
		\item $k = 2 $, $r \ge 2$ and  $(\bar{a}, \bar{b}) \in \{\overline{9}, \overline{18}\}\times \{\overline{1}, \overline{10}, \overline{19}\} \cup \{\overline{0}\}\times\{\overline{10}, \overline{19}\}$ in $(\Z/{27\Z})^2$,  
		\item $k \ge 3$, $r \ge 2$ and   $(\bar{a}, \bar{b})= \{\overline{9}, \overline{18}\}\times \{ \overline{8}, \overline{17}\}$  in $(\Z/{27\Z})^2$,  
	\end{enumerate}
	then $K$ is not monogenic.
\end{corollary}

\begin{theorem}\label{dn2}
	For $F(x) = x^n+ ax+b \in \Z[x]$, let $p$ be an odd rational prime integer such that $p \mid b$, $p \mid (n-1)$ and  $p \nmid a$. Set $n-1 = u\cdot p^k$, where $p \nmid u$. Let $\sigma = \n_p(b)$, $\rho = \n_p(a^{p-1}-1)$, $\t = \min (\sigma, \rho)$, and $\kappa = \min ( \t , k+1)$. If one of the following conditions holds:
	\begin{enumerate}
		\item $\t \neq k+1$ and $\kappa > \frac{N_p(m)}{N_p(m,u,a)}$ for some $m>1$,		
		\item $p < \sigma \cdot  N_p(1,u,a)  + 1 $ and $\sigma < \min(\rho,k+1)$,
		\item  $p < \rho \cdot  N_p(1,u,a)  + 1 $ and $\rho < \min(\sigma,k+1)$,
		\item 	 $p <(k+1) \cdot  N_p(1,u,a) + 1 $ and $k+1 < \t$, 
			\end{enumerate}
	then $K$ is not monogenic.
	
\end{theorem}
As a consequence of Theorem \ref{dn2}, the following corollary  gives explicitly certain  infinite families of non-monogenic number fields defined by irreducible trinomials  of degree $2^s \cdot 3^r+1$.   
\begin{corollary}\label{corn12}
	For $F(x) = x^{2^s \cdot 3^k+1} + ax +b \in \Z[x]$. If one of the following conditios holds:
\begin{enumerate}
\item  $s=0$, $k \ge 3$ and  $(\bar{a},\bar{b}) \in \{\overline{1},\overline{80}\}\times  \{\overline{27}, \overline{54}\}\cup \{\overline{26}, \overline{28}, \overline{53}, \overline{55}\} \times \{ \overline{0} \}$ in $(\Z/{81\Z})^2$,	
\item  $s=0$, $k \ge 4$ and  $(\bar{a},\bar{b}) \in \{\overline{1},\overline{242} \} \times \{\overline{81}, \overline{162}\}\cup \{\overline{80},  \overline{82}, \overline{161},  \overline{163}\} \times \{ \overline{0} \}$ in $(\Z/{243\Z})^2$, 
\item	$s=0$, $k =2$, $a \equiv \pm 1 \md{81} $ and $b\equiv 0 \md{81}$,
\item	$s=0$, $k =3$, $a \equiv \pm 1 \md{243} $ and $b\equiv 0 \md{243}$,
\item $s \ge 1$, $k \ge 2$ and   $(\bar{a},\bar{b}) \in \{\overline{26}\} \times \{ \overline{9}, \overline{18}\}\cup \{\overline{8}, \overline{17}\} \times \{ \overline{0} \}$ in $(\Z/{27\Z})^2$,
\item $s \ge 1$, $k \ge 3$ and $(\bar{a},\bar{b}) \in \{\overline{80} \,\, ;  \,\, \overline{27}, \overline{54}\}\cup \{\overline{26}, \overline{53} \,\,;  \,\, \overline{0},  \overline{27},  \overline{54} )$ in $(\Z/{81\Z})^2$,
\item $s \ge 1$, $k  = 1$, $a \equiv -1 \md {27}$ and $b\equiv 0 \md {27}$,
\item $s \ge 1$, $k  = 2$, $a \equiv -1 \md {81}$ and $b\equiv 0 \md {81}$,
 \item $s=2$, $k=3$, $a \equiv -1 \md{243}$ and $b\equiv 0 \md{243}$,
  \item $s=2$, $k\ge 4$ and  $(\bar{a},\bar{b}) \in \{\overline{80}, \overline{161}, \overline{242}\} \times \{\overline{81}, \overline{162} \}\cup \{\overline{80}, \overline{161}\} \times \{\overline{0} \}$ in $(\Z/{243\Z})^2$,
\item $ s \ge 3$, $ k \ge 2$ and  $(\bar{a},\bar{b}) \in\{\overline{8}, \overline{17}, \overline{26}\} \times \{\overline{9}, \overline{18} \}\cup \{\overline{8}, \overline{17} \} \times \{\overline{0} \}$ in $(\Z / {27\Z})^2$,
 \item $s=1$, $k=3$, $a \equiv 1 \md{243}$ and $b\equiv 0 \md{243}$,
 \item $s=1$, $k\ge 4$ and  $(\bar{a},\bar{b}) \in \{\overline{1}, \overline{82}, \overline{163}\} \times \{ \overline{81}, \overline{162} \}\cup \{\overline{82}, \overline{163} \} \times \{ \overline{0} )$ in $(\Z/{243\Z})^2$, 
 \item $s=2$, $k=1$, $a \equiv 1 \md{27}$ and $ b\equiv 0 \md{27}$,
\item $s=2$, $k\ge 2$ and $(\bar{a},\bar{b}) \in \{\overline{1},\overline{10}, \overline{19}\}\times\{\overline{9}, \overline{18} \}\cup \{\overline{10}, \overline{19}\}\times\{ \overline{0}\}$ in $(\Z/{27\Z})^2$,  
\end{enumerate}	
then $K$ is not monogenic.	
		
\end{corollary}
	Notice that, 	Theorems  \ref{dp^r}, \ref{dn1}, \ref{dn2}, and the main result of \cite{SK} and \cite{Smt}   does not cover the monogenity of quintic number  fields defined by $x^5+ax+b$ when $\Z_K\neq \Z[\th]$.  The following theorem gives a special study to theses number fields. 
\begin{theorem} \label{d51}
	 If one of the following conditions holds:
	\begin{enumerate}
		\item $a\equiv 1 \md 4$ and $b\equiv 2  \md 4$,
		\item   $(\bar{a},\bar{b})= ( \bar{7}, \bar{8}) \,\,\mbox{or}\,\, (\bar{15}, \bar{0})$ in $(\Z/{16\Z})^2$,		
		\item  $(\bar{a},\bar{b}) = (\overline{19},\overline{4}) \, \mbox{or}\,\,(\overline{3},\overline{20})$ in  $(\Z/{32\Z})^2$,		
			\item	 $(\overline{a},\overline{b})= ( \overline{3}, \overline{4}),  (\overline{35}, \overline{36}), ( \overline{19}, \overline{20}) \,\, \mbox{or}\,\,  (\overline{51}, \overline{52})$ in $(\Z/{64\Z})^2$,
	\item  $(\bar{a},\bar{b}) = (\overline{3},\overline{12}) \, \mbox{or}\,\,(\overline{19},\overline{28})$ in  $(\Z/{32\Z})^2$,
	\item  $(\overline{a},\overline{b})= ( \overline{3}, \overline{60}),  (\overline{19}, \overline{44}), ( \overline{35}, \overline{28}) \,\, \mbox{or}\,\,  (\overline{51}, \overline{12})$ in $(\Z/{64\Z})^2$,
			\item   $a \equiv 4 \md 8$ and $b \equiv 0 \md{8} $,
	\end{enumerate}
then $K$ is not monogenic.
\end{theorem}

 The following theorem gives explicitly certain infinite families of non-monogenic sextic number fields defined by $x^6+ax+b$.

\begin{theorem}\label{d61}
 If one of the following conditions holds:
	\begin{enumerate}
		\item $a \equiv 0 \md 8$ and $b \equiv 7 \md 8$,
		\item  $a=2 \md 4$, $b \equiv 1 \md 4$ and $\n_2(1+a+b) = 2 \n_2(a+6)$. In particular if  $(\overline{a},\overline{b})= ( \overline{6}, \overline{9}),  (\overline{14}, \bar{1}), ( \overline{22}, \overline{25}) \,\, \mbox{or} \,\, (\overline{30}, \overline{17})$ in $(\Z/{64\Z})^2$,
		\item   $a \equiv 0 \md 8$ and $b \equiv 3 \md 8$,
		\item   $a\equiv 0 \md 9$ and $b \equiv -1 \md 9$,
	\end{enumerate}
then $K$ is not monogenic.
\end{theorem}
\section{Preliminaries}

For any $\eta \in \widehat{\Z}_K$, denote by $ ( \Z_K; \Z [\eta])$ the index of $\eta$ in $\Z_K$, where $\Z[\eta]$ is the $\Z$-module generated by $\eta$. It is well known  \cite[Proposition 2.13]{Na} that:   
\begin{eqnarray}\label{indexdiscrininant}
 D (\eta) = ( \Z_K; \Z [\eta])^2 D_K,
\end{eqnarray}
where $D(\eta)$ is the discriminant of the minimal polynomial of $\eta$ and $D_K$ is the absolute discriminant of $K$. In 1878, Dedekind gave the explicit factorization of $p\Z_K$ when $p \nmid  ( \Z_K; \Z [\eta])$ (see  \cite{Co}, \cite{R}, \cite[Theorem 4.33]{Na}). He also  gave  a criterion known as   Dedekind's  criterion to test  whether  $p$ divides  or not the index $(\Z_K:\Z[\eta])$ (see \cite[Theorem 6.14]{Co}, \cite{R}, \cite{Na}).
Let 
\begin{eqnarray}\label{i(K)}
i(K) = \gcd \ \{ ( \Z_K; \Z [\eta]) \, |\,\eta \in \widehat{\Z}_K  \}
\end{eqnarray} be the index  of the field  $K$. A rational prime integer $p$ dividing $ i(K)$ is called a prime common index divisor of $K$.  If $K$ is monogenic, then $i(K) = 1$. Thus a field possessing a  prime common index divisor is not monogenic. The existence of common index divisor was first established by R. Dedekind. He used Dedekind's criterion and his factorization  theorem to show that the cubic number field $K=\Q(\th)$, where $\th$ is a root of $x^3+x^2-2x+8$ cannot be monogenic, since the prime  $2$ splits completely in $\Z_K$.  Further, in \cite{R},  Dedekind gave a necessary and sufficient condition  on prime $p$ to be a common index divisor of $K$. This condition depends upon the factorization of the prime  $p$ in $\Z_K$ (see also \cite{He, Hen}).  E. Zylinski \cite{Zylinski} showed that, if $p$ divides $i(K)$ then $p < n$. When $p \nmid i(K)$; there exist $\eta \in \widehat{\Z}_K$ such that $p\nmid  ( \Z_K; \Z [\eta])$. Then,  by Dedekind's theorem,  we explicitly  factorize  $p\Z_K$; it is  analogous to the factorization of the minimal polynomial $P_{\eta}(x)$ of $\eta$  modulo $p$. But if $p$ divides the index $ i(K)$, then Dedekind's factorization  theorem is not applicable. Hensel \cite{Hecor}, proved that the prime ideals of $\Z_K$ lying above $p$ are in one-to-one correspondence with irreducible factors of $f(x)$ in $\Q_p(x)$. In 1928,   
O. Ore \cite{O} developed a method for factoring $f(x)$ in $\Q_p(x)$, factoring $p$ in $\Z_K$ when $f(x)$ is $p$-regular. The method based on Newton polygon techniques. Now, we  recall   some fundamental facts on  Newton polygon techniques  applied on prime ideal factorization. {For more details, we refer to \cite{El, EMN, Nar,  MN92, O}}.
Let $p$  be a rational prime integer and $\n_p$  the discrete valuation of $\Q_p(x)$  defined on $\Z_p[x]$ by $\n_p(\sum_{i=0}^{m} a_i x^i) = \min \{ \n_p(a_i), \, 0 \le i \le m\}$. Let $\phi \in \mathbb{Z}[x]$ be a monic polynomial whose reduction modulo $p$ is irreducible. Any monic irreducible polynomial $f(x) \in \mathbb{Z}[x]$ admits a unique $\phi$-adic development $ f(x )= a_0(x)+ a_1(x) \phi(x) + \cdots + a_n(x) {\phi(x)}^n,$ with $ deg \ ( a_i (x))  < deg \ ( \phi(x))$. For every $0\le i \le n,$   let $ u_i = \n_p(a_i(x))$. The $\phi$-Newton polygon of $f(x)$ is the  lower boundary convex envelope of the set  of points $ \{  ( i , u_i) \, , 0 \le i \le n \, , a_i(x) \neq 0  \}$ in the Euclidean plane, which we denote   by $N_{\phi} (f)$. The polygon  $N_{\phi} (f)$ is the union of different adjacent sides $ S_1, S_2, \ldots , S_g$ with increasing slopes $ \lambda_1, \lambda_2, \ldots,\lambda_g$. We shall write $N_\phi(f) = S_1+S_2+\cdots+S_g$. The polygon determined by the sides of negative slopes of $N_{\phi}(f)$ is called the  $\phi$-principal Newton polygon of $f(x)$ and well denoted by $\npp{f}$. The length of $\npp{f}$ is $ l(\npp{f}) = \nu_{\overline{\ph}}(\overline{f(x)})$;  the highest power of $\phi$ dividing $f(x)$ modulo $p$.\\
Let $\mathbb{F}_{\phi}$ be the finite field   $ \mathbb{Z}[x]\textfractionsolidus(p,\phi (x)) \simeq \mathbb{F}_p[x]\textfractionsolidus (\overline{\ph}) $ (note that if $deg(\ph)=1$, then $\F_{\ph} \simeq \F_p  $).
We attach to  any abscissa $ 0 \leq i \leq  l(\npp{f})$ the following residual coefficient $ c_i \in  \mathbb{F}_{\phi}$ as follows:

$$c_{i}=
\left
\{\begin{array}{ll} 0,& \mbox{if }  (i , u_i ) \, \text{ lies  strictly  above }  \ \npp{f},\\
\left(\dfrac{a_i(x)}{p^{u_i}}\right)
\,\,
\md{(p,\phi(x))},&\mbox{if } \  (i , u_i ) \, \text{lies on } \npp{f}. 
\end{array}
\right.$$ Now, let $S$ one of the sides of $\npp{f}$ and $\lambda = - \frac{h}{e} $ be its slope, where $e$ and $h$ are two positive coprime integers. The length of $S$, denoted $l(S)$ is the length of its  projection to the horizontal axis. The degree of $S$ is $ d = d(S) = \frac{l(S)}{e}$; it is equal to the the number of segments into which the integral lattice divides $S$. More precisely, if $ (s , u_s)$ is the initial point of $S$, then the points with integer coordinates  lying in $S$ are exactly $  (s , u_s) ,\ (s+e , u_s - h) , \ldots, (s+de , u_s - dh)$. We attach to $S$ the following residual polynomial defined by $ R_{\l}(f)(y) = c_s + c_{s+e}y+ \cdots + c_{s+(d-1)e}y^{d- 1}+ c_{s+de}y^d \in \mathbb{F}_{\phi}[y]$.  As defined in \cite[Def. 1.3]{EMN},   the $\ph$-index of $f(x)$, denoted by $ind_{\ph}(f)$, is  deg$(\ph)$ times the number of points with natural integer coordinates that lie below or on the polygon $\npp{f}$, strictly above the horizontal axis  and strictly beyond the vertical axis (see FIGURE $1$). We say that the polynomial $f(x)$ is $\phi$-regular with respect to $p$ if for each side $S$ of $\npp{f}$, the associated residual polynomial $ R_{\l}(f)(y)$ is separable in $\mathbb{F}_\phi[y]$. The polynomial $f(x)$ is said to be $p$-regular if $f(x)$ is $\phi_i$-regular for every $ 1 \leq i \leq t$ , where $\overline{f(x)}=\prod_{i=1}^t\overline{\ph_i}^{l_i}$is the factorization  of $\overline{f(x)}$ into a product  of powers of distinct irreducible polynomials in  $\mathbb{F}_p [x]$. For every $i=1,\dots,t$, let  $N_{\ph_i}^+(f)=S_{i1}+\dots+S_{ir_i}$ and for every {$j=1,\dots, r_i$},  let $R_{\l_{ij}}(f)(y)=\prod_{s=1}^{s_{ij}}\psi_{ijs}^{n_{ijs}}(y)$ be the factorization of $R_{\l_{ij}}(f)(y)$ in $\F_{\ph_i}[y]$. 
By theorem of the product, theorem of the polygon and theorem of the residual polynomial (see \cite[Theorems 1.13, 1.15 and 1.19]{Nar}),    we have the following  theorem of Ore, which plays a significant role in the proof of our theorems (see \cite[Theorem 3.9]{El},  \cite[Theorem 1.7 and Theorem 1.9]{EMN},  \cite{MN92} and \cite{O}):
\begin{theorem}\label{ore} (Theorem of Ore)
	\begin{enumerate}
		\item
		$\nu_p(ind(f))=\nu_p((\Z_K:\Z[\th]))\ge \sum_{i=1}^t ind_{\ph_i}(f)$ and  equality holds if $f(x)$ is $p$-regular;  every $n_{ijs}= 1$.
		\item
		If  $f(x)$ is $p$-regular, then 
		$$p\Z_K=\prod_{i=1}^t\prod_{j=1}^{r_i}
		\prod_{s=1}^{s_{ij}}\pF^{e_{ij}}_{ijs},$$ where $e_{ij}$ is the ramification index
		of the side $S_{ij}$ and $f_{ijs}=\mbox{deg}(\ph_i)\times \mbox{deg}(\psi_{ijs})$ is the residue degree of $\mathfrak{p}_{ijs}$ over $p$.
	\end{enumerate}
\end{theorem}

\begin{example} Consider the monic  irreducible polynomial $f(x)=x^4-4x^3+12x^2-8x+95$. Then $f(x)\equiv \ph^4 \md 2 $, where $\ph = x-1$. The $\ph$-adic development of $f(x)$ is $$f(x)= \ph^4+ 6\ph^2+8\ph+96.$$ Thus $\npp{f} = S_1+S_2$ with respect to $\n_2$ has two sides, with $d(S_1) =2$, $d(S_2) = 1$, $\l_1 = -2$ and $\l_2 = \frac{-1}{2}$ (see FIGURE $1$).

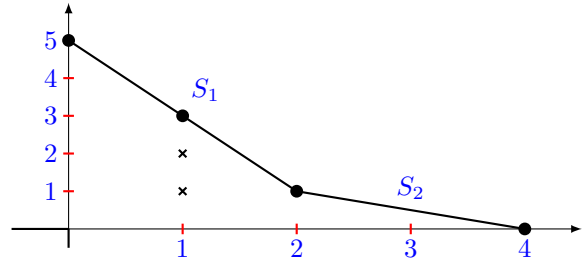
\begin{figure}[htbp]

	\centering
	
	\begin{tikzpicture}[x=1.5cm,y=0.5cm]
	\draw[latex-latex] (0,6) -- (0,0) -- (4.5,0) ;

	\draw[thick] (0,0) -- (-0.5,0);
	\draw[thick] (0,0) -- (0,-0.5);
	
	\draw[thick,red] (1,-2pt) -- (1,2pt);
	\draw[thick,red] (2,-2pt) -- (2,2pt);
	\draw[thick,red] (3,-2pt) -- (3,2pt);
	\draw[thick,red] (4,-2pt) -- (4,2pt);
	\draw[thick,red] (-2pt,1) -- (2pt,1);
	\draw[thick,red] (-2pt,2) -- (2pt,2);
	\draw[thick,red] (-2pt,3) -- (2pt,3);
	\draw[thick,red] (-2pt,4) -- (2pt,4);	
	\draw[thick,red] (-2pt,5) -- (2pt,5);	
	\node at (1,0) [below ,blue]{\footnotesize  $1$};
	\node at (2,0) [below ,blue]{\footnotesize $2$};
	\node at (3,0) [below ,blue]{\footnotesize  $3$};
	\node at (4,0) [below ,blue]{\footnotesize  $4$};
	\node at (0,1) [left ,blue]{\footnotesize  $1$};
	\node at (0,2) [left ,blue]{\footnotesize  $2$};
	\node at (0,3) [left ,blue]{\footnotesize  $3$};
	\node at (0,4) [left ,blue]{\footnotesize  $4$};
	\node at (0,5) [left ,blue]{\footnotesize  $5$};
	\draw[thick, mark = *] plot coordinates{(0,5) (1,3) (2,1) (4,0)};
	\draw[thick, only marks, mark=x] plot coordinates{(1,1) (1,2)  (2,1)     };
	\node at (1.2,3.1) [above  ,blue]{\footnotesize $S_{1}$};
	\node at (3,0.5) [above   ,blue]{\footnotesize $S_{2}$};
	\end{tikzpicture}
	\caption{ The $\ph$-principal Newton polygon  $\npp{f}$ with respect to $\n_2$.}
\end{figure}The residual polynomials attached to the sides of $\npp{f}$ are
$R_{\l_1}(f)(y) = 1 + y + y^2$ and $R_{\l_2}(f)(y)=   1 +y $, which are irreducible polynomials in $\F_{\ph}[y] \simeq \F_2[y] $. Thus $f(x)$ is $\ph$-regular, hence it is $2$-regular. By  Theorem \ref{ore},  $\nu_2(ind(f))=\nu_2((\Z_K:\Z[\a])) = ind_{\ph}(f) = \deg(\ph) \times  4 = 4 $ and $2\Z_K = \pF_1 \pF_2^2  $, with respective residue degrees $ f_1=2 $ and $f_2  = 1 $.  
\end{example}

In order to prove  Theorem of the product, J. Gu\`{a}rdia, J. Montes and E. Nart introduced in \cite{Nar} the notion of $\ph$-admissible development. In this paper we will use these technique in order to treat some special cases when the $\ph$-adic development of a given polynomial  $ f(x)$ is not obvious. Let
\begin{equation}\label{admissdev}
	f(x) = \sum_{j=0}^{n}A_j(x) \ph(x)^j,  \, \,  A_j(x) \in \mathbb{Z}_p[x],
\end{equation}
be a  $\ph$-development of $f(x)$, not necessarily the $\ph-$adic one. Take $ \omega_j = \nu_p(A_j(x))$, for all $ 0 \leq j \leq n$. Let $N$ be the principal Newton polygon of the set of points $\{ ( j, \omega_j ) \, \, | \,  0 \leq j \leq n, \omega_j \neq \infty \}$. To any $ 0 \leq j \leq n$, we attach a residual coefficient as follow :  $$c^{'}_{j}=
\left
\{\begin{array}{ll} 0,& \mbox{ if } (j,{\it \omega_j}) \mbox{ lies strictly
	above } N
,\\
\left(\dfrac{A_{j}(x)}{p^{\omega_j}}\right)
\,\,
\md{(p,\phi(x))},&\mbox{ if }(j,{\it \omega_j}) \mbox{ lies on }N.
\end{array}
\right.$$
Moreover,  for any side $S$ of $N$ with slope $\l$, we  define the residual polynomial associated to $S$ and noted $R_{\l}^{'}(f)(y)$ (similar to the residual polynomial $R_{\l}(F)(y)$ defined from the $\ph$-adic development). We say that a $ \ph$-development  (\ref{admissdev}) of $f(x)$ is admissible if $ c^{'}_{j} \neq 0 $ for each abscissa $j$ of a vertex of $N$. Note that   
$ c^{'}_{j} \neq  0 $ if and only if $ \overline{\ph(x)} \nmid \overline{\left(\dfrac{A_{j}(x)}{p^{\omega_j}}\right)}.$  For more details, see \cite{Nar}.  
\begin{lemma} \label{admiss}$($\cite[Lemma $1.12$]{Nar}$)$ \\If a $ \ph$-development of $f(x)$ is admissible, then $ \npp {f} = N$ and $ c^{'}_j = c_j $. In particular, for any segment $S$ of $N$ with slope $\l$ we have $R_{\l}^{'}(f)(y) =R_{\l}(f)(y)$.  
\end{lemma}

When the polynomial $f(x)$ is not $p$-regular; certain factors of $f(x)$  provided by certain factors of certain  residual polynomials $R_{\l_{ij}}(f)(y)$ are not irreducible in $\Q_p(x)$.  Montes,  	Nart and  Gu\'{a}rdia are recently  introduced an efficient algorithm  to factorize completely  the principal ideal  $p\Z_K$ (see \cite{ Nar, MN92}). They defined the Newton polygon of order $r$ and they proved an extension of the theorem of the product, theorem of the polygon, theorem of the residual polynomial  and theorem of index in order $r$. As we will use this algorithm in second order; $r=2$, we  shortly recall those concepts that we use throughout. Let $\ph$ be a monic irreducible factor of $f(x)$ modulo $p$.  Let $S$ be  a side of $N_1 = N^{+}_{\ph}(F)$, with slope $\lambda = - \frac{h}{e}$, with $h$ and $e$ are two coprime positive integers such that the associated residual polynomial  $\psi_1(y)=R_{\l}(f)(y)$ is of degree $f \ge 2$ and not separable in $\F_{\ph}[y]$. A type of order $2$ is a chain:
$$ (\ph(x); \lambda, \ph_2(x); \lambda_2, \psi_2(y)),$$ where  $\ph_2(x)$ is a monic irreducible polynomial in $\Z_p[x]$ of degree $m_2 = e \cdot f \cdot \deg(\ph)$, $\lambda_2$ is a negative rational number and $\psi_2(y) \in \F^2 = \F_{\ph}[y]/{(\psi_1(y))} $ such that
\begin{enumerate}
	\item  $N_1(\ph_2) = N^{+}_{\ph}(\ph_2)$ is one-sided with slope $\lambda$.
	\item The residual polynomial in  order $1$ of $\ph_2$;  $R_{\l}(\ph_2)(y) \simeq \psi_1(y)$ in  $\F_{\ph}[y]$.
	\item $\l_2$ is a slope of certain side of $\ph_2$-Newton polygon of second order (To be specified below) and $\psi_2(y) = R^2_{\l_2}(f)(y)$ is the associated residual polynomial of second order. 
\end{enumerate}The key polynomial $\ph_2$ induces a valuation  $\n^{2}_p$ on $\Q_p(x)$, called the augmented valuation of $\n_p$ of second order with respect to $\ph$ and $\lambda$. By  \cite[Proposition 2.7]{Nar}, If $P(x) \in \Z_p[x]$ such that  $P(x )= a_0(x)+ a_1(x) \ph(x) + \cdots + a_l(x) {\ph(x)}^l$. Then $\n^{2}_p(P(x)) = e \times  \min\limits_{0 \le j \le l}\{ \n_p(a_i(x)) + i (\n_p(\ph(x))+ |\lambda|)\}$, in particular $\n^{2}_p(\ph_2(x)) = e \cdot f \cdot \n_p(\ph(x))$. Let $f(x)= a_0(x)+ a_1(x) \ph_2(x) + \cdots + a_t(x) {\ph_2(x)}^t$ be the $\ph_2$-adic development of $f(x)$ and let $\mu_i = \n^{2}_p(a_i(x) \ph_2(x)^i) $ for every $0 \le i \le t$. The $\ph_2$-Newton polygon of $f(x)$  of second order with respect $\n_p^2$ is the lower boundary of the convex envelope of the set of points   $ \{  ( i , \mu_i) \, , 0 \le i \le t  \}$ in the Euclidean plane, which we denote   by $N_{2} (f)$.  We will use theorem of the polygon and theorem of residual polynomial in second order (see \cite[Theorem 3.1 and 3.4]{Nar} for more general treatment).

For the determination of certain Newton polygons, we will need to evaluate the $p$-adic valuation of the binomial coefficient $\dbinom{p^r}{j} $, this is the object of the following well knowing lemma. For the proof,  see for example \cite{BFC}. 
\begin{lemma} \label{binomial}
	Let $p$ be a rational prime integer and $r$ be a positive integer. Then\begin{eqnarray*}
	\nu_p\left(\binom{p^r}{j}\right)  =  r - \nu_p(j) 
	\end{eqnarray*}for any integer $j= 1,\dots,p^r-1 $. 
\end{lemma}

 The following lemma gives a sufficient  condition for a rational prime integer $p$ to be a prime common index divisor of the field $K$. For the proof, see \cite{R} and \cite[Theorems 4.33 and 4.34 ]{Na}.
\begin{lemma} \label{comindex}
	Let  $p$ be  rational prime integer and $K$ be a number field. For every positive integer $m$, let $P_m$ be the number of distinct prime ideals of $\Z_K$ lying above $p$ with residue degree $m$ and  $N_p(m)$ be the number of monic irreducible polynomials of  $\F_p[x]$ of degree $m$. If $ P_m > N_p(m)$ for some positive integer $m$, then $p$ is a prime common index divisor of $K$.
\end{lemma}
To apply the last lemma, one needs to know the number $N_p(m)$ of monic irreducible polynomials over $\F_p$ of degree $m$ which  is given by the following proposition. 
\begin{proposition}$($\cite[Proposition 4.35]{Na}$)$
	The number of  monic irreducible polynomials of degree $m$ in $\F_p[x]$ is given by:
	\begin{eqnarray*}
		N_p(m) = \frac{1}{m} \sum_{d \mid m} \mu (d) p^{\frac{m}{d}},
	\end{eqnarray*}where $\mu$ is the M\"{o}ubius function. 
\end{proposition}

\section{Proofs of main results}  
{\begin{proof}[Proof of Theorem \ref{mono}]\
		\\ Since $F(x) \equiv \ph^{p^r} \md p$, where $\ph = x$ and $\nph{F}=S$  has  a single side of degree $1$ (because gcd$(u,p^r)=1$, we conclude that $R_{l}(F)(y)$ is irreducible over $\F_{ \ph}$, and so  $F(x)$ is irreducible over $\Q_p$. Let $L=\Q_p(\th)$ and $K=\Q_p(\th)$. Since $\Q_p$ is a Henselianfield, there is a unique  valuation $\om$  of $L$ extending $\nu_p$.
		  Let  $(x,y)\in \Z^2$  be the unique solution of the diophantine equation $xu-yp^r=1$ with $0\le x <p^r$ and $\eta=\frac{\th^x}{p^y}$. Let us show that $\eta\in \Z_K$ and $\Z_K=\Z[\eta]$. First, by definition, $\eta\in K$. By \cite[Corollary 3.1.4]{En}, in order to show that $\eta\in \Z_K$, we need to show that $\om(\eta)\ge 0$. Since $\nph{F}=S$ has  a single side of slope $-u/p^r$, we conclude that $\om(\th)=u/p^r$, and so $\om(\eta)=x\frac{u}{p^r}-y=\frac{xu-yp^r}{p^r}=\frac{1}{p^r}$. Since $x$ and $p^r$ are coprime, we conclude that $K=\Q(\eta)$. Let $g(x)$ be the minimal polynomial of $\eta$ over $\Q$. By the formula relating roots and coefficients of a monic polynomial, we conclude that $g(x)=x^{p^r}+\sum_{i=1}^{p^r}(-1)^is_ix^{p^r-i}$, where $s_i=\displaystyle\sum_{k_1<\dots<k_i}\eta_{k_1}\cdots\eta_{k_i}$, where $\eta_{1},\dots, \eta_{p^r}$ are the $\Q_p$-conjugates of $\eta$. Since there is a unique valuation extending $\nu_p$ to any algebraic extension of $\Q_p$, we conclude that $\om(\eta_i)=u/p^r$ for every $i=1,\dots,p^r$. Thus $\nu_p(s_{p^r})\om(\eta_{1}\cdots\eta_{p^r}=p^r\times 1/p^r=1$ and $\nu_p(s_{p^r})\ge i/p^r$ for every $i=1,\dots, p^r-1$. That means that $g(x)$ is a $p$-Eseinstein polynomial. Hence $p$ does not divide the index $(\Z_K:\Z[\eta])$. As by hypothesis $p$ is the unique positive prime integer such that $p^2$ divides $D(\th)$ and by definition of $\eta$, $p$ is the unique positive prime integer candidate to divide  $(\Z[\th]:\Z[\eta])$, we conclude that for every prime integer $q$, $q$ does not divide $(\Z_K:\Z[\eta])$, which means that $\Z_K=\Z[\eta]$.
\end{proof}}
In every case, we prove that $\Z_K$ has no power integral basis by finding an adequate rational prime integer $p$ which is a common index divisor of $K$. For this reason, in view of Lemma \ref{comindex},  it suffice that the prime ideal factorization of  $p\Z_K$  satisfies  the inequality $ P_m > N_p(m)$ for some positive integer $m$.
As proved in \cite{BFC}, we will use  frequently without indicating the fact that when a rational prime integer $p$ does not divide a rational integer $b$, then  $\n_p((-b)^{p^k}+b)=\n_p(b^{p-1}-1)$ for every positive rational integer $k$.
	\begin{proof}[Proof of Theorem \ref{dp^r}]\
		\\
	Since $p\mid a$ and $p\nmid b$, we have $F(x) \equiv \ph^{p^r} \md p$, where $\ph = x+b$. Write 
	\begin{eqnarray}\label{devp^r}
		F(x)&=& (x+b-b)^{p^r}+ax+b=(\ph-b)^{p^r}+a(\ph - b )+b \nonumber \\
		&=& \ph^{p^r}+ \sum_{j=1}^{p^r-1} (-1)^{j+1} \binom{p^r}{j}  b^{p^r-j}\ph^j+ a(\ph - b) +b +(-b)^{p^r} \nonumber \\
			&=& \ph^{p^r}+ \sum_{j=2}^{p^r-1}(-1)^{j+1}\binom{p^r}{j} b^{p^r-j}\ph^j+ (a+p^r \cdot b^{p^r-1} )\ph +b +(-b)^{p^r}-ab.
	\end{eqnarray}	
	Since  $a \equiv 0 \md{p^{p+1}}$, $b^{p-1}\equiv 1 \md{p^{p+1}}$, and $r \ge p$, we conclude that $ \mu = \n_p(a+ p^r \cdot b^{p^r-1}) \ge \min(r,p+1)$ and $\n =\n_p(b +(-b)^{p^r}-ab) \ge p+1$. By Lemma \ref{binomial} and the $\ph$-adic development (\ref{devp^r})
 of $F(x)$, if  $\mu > \n $, then the $\ph$-principal Newton polygon of $F(x)$ with respect  to $\n_p$,  $\npp{F}= S_1+ \cdots +  S_{t-(p-1)}+ \cdots + S_t$ has $t$ sides of degree $1$ each with $t \ge p+1 $; the segment $S_{t-k}$ is the segment joining the points  $(p^{r-k-1},k+1)$ and $(p^{r-k},k)$ with  slopes $\l_{t-k} =\frac{-1}{(p-1)p^{r-k-1}}$, with ramification indices $e_{t-k} = (p-1)p^{r-k-1}$ for every $k = 0, 1, \ldots, p-1$, and the segment $S_1$ has $(0, \n)$ as the first  point  and $(1, \mu )$ as the end point with  ramification index $e_p = 1$ (see FIGURE 2). Thus $R_{\l_{j}}(F)(y)$ is irreducible over $\F_{\ph} \simeq \F_p$ for every $0 \le j \le t$ as it is of degree $1$. So, the polynomial $F(x)$ is $p$-regular. By Theorem \ref{ore}, $p\Z_K =  \prod_{k=0}^{p - 1 } \pF_k^{e_{t-k}} \cdot \pF_p^{e_p} \cdot  \aF$ for some non-zero ideal $\aF$ of $\Z_K$, with $f(\pF_k/p)= 1$ for every $0 \le k \le p$. If $\mu \ge \nu$. By Lemma \ref{binomial}, and (\ref{devp^r}),  $\npp{F}= S_1+ \cdots +  S_{t-p}+ \cdots + S_t$ has $t$ sides with $t \ge p+1 $, and the last $p+1$ sides are all of degree $1$; $S_{t-k}$ is the segment joining the points  $(p^{r-k-1},k+1)$ and $(p^{r-k},k)$ for every $k = 0, 1, \ldots, p$. By Theorem \ref{ore}, $p\Z_K =  \prod_{k=0}^{p } \pF_k^{e_{t-k}}\cdot  \bF$ for some  non-zero ideal   $\bF$  of $\Z_K$ . So, for the rational  prime  integer $p$, we have $P_1 \ge p+1 > N_p(1) = p$. Thus by Lemma \ref{comindex}, $p\mid i(K)$ and so $K$ is not monogenic.
 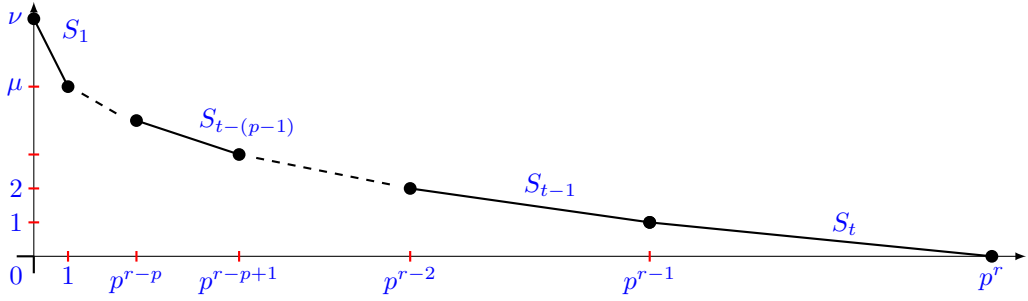
\begin{figure}[htbp] 
 	 \centering
 	\begin{tikzpicture}[x=0.45cm,y=0.45cm]
 	\draw[latex-latex] (0,7.5) -- (0,0) -- (29,0) ;
 	\draw[thick] (0,0) -- (-0.5,0);
 	\draw[thick] (0,0) -- (0,-0.5); 
 	\draw[thick,red] (1,-2pt) -- (1,2pt);
 	\draw[thick,red] (3,-2pt) -- (3,2pt);
 	\draw[thick,red] (6,-2pt) -- (6,2pt);
 	\draw[thick,red] (11,-2pt) -- (11,2pt);
 	\draw[thick,red] (18,-2pt) -- (18,2pt);
 		\draw[thick,red] (28,-2pt) -- (28,2pt);
 	\draw[thick,red] (-2pt,1) -- (2pt,1);
 	\draw[thick,red] (-2pt,2) -- (2pt,2);
 	\draw[thick,red] (-2pt,3) -- (2pt,3);
 	\draw[thick,red] (-2pt,5) -- (2pt,5);
 		\draw[thick,red] (-2pt,7) -- (2pt,7);
 	\node at (0,0) [below left,blue]{\footnotesize  $0$};
 	\node at (1,0) [below ,blue]{\footnotesize  $1$};
 	\node at (3,0) [below ,blue]{\footnotesize $p^{r-p}$};
 	\node at (6,0) [below ,blue]{\footnotesize  $p^{r-p+1}$};
 	\node at (11,0) [below ,blue]{\footnotesize  $p^{r-2}$};
 	\node at (18,0) [below ,blue]{\footnotesize  $p^{r-1}$};
 		\node at (28,0) [below ,blue]{\footnotesize  $p^{r}$};
 	\node at (0,1) [left ,blue]{\footnotesize  $1$};
 	\node at (0,2) [left ,blue]{\footnotesize  $2$};
 	\node at (0,5) [left ,blue]{\footnotesize  $\mu$};
 	 	\node at (0,7) [left ,blue]{\footnotesize  $\n$};
 	\draw[thick,mark=*] plot coordinates{(0,7) (1,5)};
 		\draw[thick,mark=*] plot coordinates{(3,4) (6,3)};
 				\draw[thick,mark=*] plot coordinates{(11,2) (18,1)};
 					\draw[thick,mark=*] plot coordinates{(28,0) (18,1)};
 						\draw[thick, dashed] plot coordinates{(1.4,4.8) (2.6,4.1) };
 							\draw[thick, dashed] plot coordinates{(6.5,2.9) (10.5,2.1) };
 	\node at (0.5,6) [above right  ,blue]{\footnotesize  $S_{1}$};
 	\node at (4.5,3.2) [above right  ,blue]{\footnotesize  $S_{t-(p-1)}$};
 	\node at (14,1.4) [above right  ,blue]{\footnotesize  $S_{t-1}$};
 	\node at (23,0.3) [above right  ,blue]{\footnotesize  $S_{t}$};
 	\end{tikzpicture}
 	\caption{    \large  $\npp{F}$ at odd prime $p$ with $a \equiv 0 \md{p^{p+1}}$, $b^{p-1}\equiv 1 \md{p^{p+1}}$, $r \ge p$ and $\n > \mu$ .\hspace{5cm}}
 \end{figure}
\end{proof}

   For the proof of Theorems \ref{dn1} and \ref{dn2} we will use the following technical result.  
\begin{lemma}\label{lemtech}
	Let $p$ be a rational prime integer and $f(x) \in \Z[x] $ be a polynomial which is separable modulo $p$. Let $g(x)$ be a monic irreducible factor of $\ol{f(x)}$ in $\F_p[x]$. Then, we can select a monic lifting $\ph \in \Z[x]$ of $g(x)$ such that  $ f(x)= \ph(x) U(x)+pT(x)$ for some polynomials $U(x)$, $T(x) \in \Z[X]$ such that $g(x) \nmid \overline{U(x)}$, $\overline{T(x)} \neq \overline{0}$ and $\deg(T(x)) < \deg(g(x))$. 
\end{lemma}
\begin{proof}\
	\\As $f(x)$ is separable modulo $p$,  by the Euclidean division algorithm, we set $f(x) = \ph(x) U_1(x)+p^l R_1(x)$, where $l$ is a positive rational  integer, $U_1(x), R_1(x) \in \Z[x]$ such that   $ \deg(R_1(x)) < \deg(\ph(x))$ and $\overline{\ph(x)} \nmid \overline{U_1(x)}$. If $l\ge 2$, then write $f(x)= (\ph(x) -p+p)U_1(x) + p^l R_1(x) $. Set $\ph^{\star}(x)= \ph(x)-p$. Using the Euclidean  division algorithm again, we see that $U_1(x) = \phi^{\star}(x) U_2(x) + R_2(x)$ for some two polynomials $ U_2(x)$ and $ R_2(x)$ in $ \Z[x]$ such that $\deg(R_2(x)) < \deg(\ph^{\star}(x))= \deg(\ph(x))$ and $\overline{R_2(x)} \neq \bar{0}$ (because $\overline{\ph(x)} \nmid \overline{U_1(x)}$). Then we have \begin{eqnarray}
	f(x)&=& \ph^{\star}(x) (U_1(x)+ p U_2(x))+ p(R_2(x) + p^{l-1}R_1(x)) \nonumber \\
	&=& \ph^{\star}(x) U(x)+ p T(x), \nonumber
	\end{eqnarray}where $U(x)= U_1(x)+ p U_2(x)$ and $T(x)=R_2(x) + p^{l-1}R_1(x) $. So, up to replace the lifting  $\ph(x)$ by $\phi^{\star}(x)$,  our claim holds.

\end{proof}

\begin{proof}[Proof of Theorem \ref{dn1}]\
\\\begin{enumerate}
\item 
Under the assumptions of  Theorem \ref{dn1}, we have $F(x) \equiv (x^s+b)^{p^r} \md p$. Since $p\nmid sb$, we have the polynomial $x^s+b$ is square free modulo $p$. Let $g(x)$ be a monic irreducible factor of  $x^s+\overline{b}$ in $\F_p[x]$ of degree $m > 1$. By Lemma \ref{lemtech},  for a suitable lifting $\ph(x)$ of $g(x)$, there exist two polynomials $U(x)$ and $T(x)\in \Z[x]$ such that  $x^s+b=\ph(x) U(x)+pT(x)$, where $\overline{\ph(x)}$ does not divide $\overline{U(x)}\overline{T(x)}$. Set $M(x)=pT(x)-b$ and write 
\begin{eqnarray*}
F(x)&=& x^n+ax+b=(x^s)^{p^r}+ax+b=(\ph(x) U(x)+M(x))^{p^r}+ax+b, \nonumber \\
&=& (\ph(x) U(x))^{p^r}+ \sum_{j=1}^{p^r-1} \binom{p^r}{j} M(x)^{p^r-j}U(x)^j \ph(x)^j + M(x)^{p^r} +ax+b.\nonumber
\end{eqnarray*}
By Binomial expansion and Lemma \ref*{binomial}, we see that  $$M(x)^{p^r}= p^{r+1}H(x)+(-b)^{p^r}, $$where $$H(x)= b^{p^r-1}T(x)+\frac{1}{p^{r+1}}\sum_{j=0}^{p^r-2}(-1)^j\binom{p^r}{j}b^j (pT(x))^{p^r-j}. $$It follows that
  \hspace{10cm} \begin{eqnarray}\label{devn1}
F(x)= (\ph(x) U(x))^{p^r}+ \sum_{j=1}^{p^r-1} \binom{p^r}{j} M(x)^{p^r-j}U(x)^j \ph(x)^j +p^{r+1} H(x)+ax+(-b)^{p^r}+b.
\end{eqnarray}
Thus $F(x) =\sum_{j=0}^{p^r} A_j(x)\ph(x)^j$, where 
\[  \begin{cases}
A_0(x) = p^{r+1} H(x)+ax+(-b)^{p^r}+b,\,\\
A_j(x) =\displaystyle \binom{p^r}{j} M(x)^{p^r-j}U(x)^j \mbox{for every } 1 \leq j \leq p^r.
\end{cases} \]
Using Lemma \ref{binomial} and  reducing modulo $p$, we get $\omega_j = \n_p(A_j(x)) = r-\n_p(j)$ and  $\overline{\left(\dfrac{A_j(x)}{p^{\omega_j}}\right)}=   \dbinom{p^r}{j}_p  \overline{M(x)}^{p^r-j}  \overline{U(x)}^j$   for every $1 \leq j \leq p^r$. Since  $\overline{M(x) }= \bar{b}\neq \bar{0} $ and  $\overline{\ph(x)} \nmid \overline{U(x)}$, then  $\overline{\ph(x)} \nmid \overline{\left(\dfrac{A_j(x)}{p^{\omega_j}}\right)}$ for every $1 \leq j \leq p^r$.   Moreover, If $\delta > r+1$, then $\omega_0 = \n_p(A_0(x))=r+1 = \omega $. On the other hand, by Lemma \ref{binomial}, \begin{eqnarray*}
\n_p(\binom{p^r}{j} \cdot b^j \cdot (pT(x))^{p^r-j}) = r-\n_p(j)+p^r-j > r+1
\end{eqnarray*}for every $0 \le j \le p^r-2$. It follows that $\overline{\left(\dfrac{A_0(x)}{p^{\omega_0}}\right)}= \overline{\left(\dfrac{A_0(x)}{p^{r+1}}\right)} = \overline{H(x)} = \overline{b^{p^r-1} T(x)}$. Then $\overline{\ph(x)} \nmid \overline{\left(\dfrac{A_0(x)}{p^{\omega_0}}\right)}$ (because $\overline{\ph(x)} \nmid \overline{T(x)}$). Consequently, the $\ph$-development  (\ref*{devn1}) of $F(x)$ is admissible. By Lemma \ref{admiss},  $\npp{F}=S_0+S_1+\cdots+S_r$ has $r+1$ sides of degree $1$ each joining the points   $ \{(0, r+1) \} \cup \{(p^j, r-j) \, | \, 0\leq j \leq r \}$ in the Euclidean plane with respective slopes $\l_1 = -1$ and $\l_k = \dfrac{-1}{e_k}$ with $e_k=(p-1)p^{k-1}$ for every $1 \le k \le r$ (see FIGURE 3  for example when $p=3$, $\delta \ge 5$ and $r=4$). Thus $R_{\l_k}(F)(y)$ is irreducible over $\F_{ \ph}$ as it is of degree $1$ for every $0 \le k \le r$. By Theorem \ref{ore}, the irreducible factor  $\overline{\ph(x)}$ of $\overline{F(x)}$ provides $r+1 $ prime  ideals above the rational prime $p$ with residue degree $ deg(\ph(x)) \times \deg(R_{\l_k}(F)(y))= m \times 1 = m$ each. Therefore, the $N_p(m,s,b)$ monic irreducible factors of $F(x)$ modulo $p$ provides $\omega \cdot N_p(m,s,b) $ prime ideals of $\Z_K$ above  $p$. By Lemma \ref{comindex}, if $\omega \cdot N_p(m,s,b) > N_p(m)$, then $p$ is a prime common index divisor of $K$. Hence $K$ is not monogenic. \\
Similarly, if $r+1 > \delta$, then $\omega_0=\delta = \omega $ and $\overline{\left(\dfrac{A_0(x)}{p^{\omega_0}}\right)}=\overline{\left(\dfrac{A_0(x)}{p^{\delta}}\right)} = a_p \bar{x}+ (b+(-b)^{p^r})_p$. So $\overline{\ph(x)} \nmid  \overline{\left(\dfrac{A_0(x)}{p^{\omega_0}}\right)}$ (because $m > 1$). It follows that  the $\ph$-development (\ref{devn1}) of $F(x)$ is admissible. and by Lemma \ref{admiss}, $\npp{F} = S_1+ S_2+\cdots+S_{\delta}$ has $\delta$ sides of degree $1$ each  joining the points $\{(0,\delta),(p^{r-\delta+1}, \delta-1),(p^{r-\delta+2}, \delta-2), \ldots,(p^r,0)\}$, with respective ramification indices $e_1=p^{r-\delta+1}$ and $e_k= (p-1)p^{r-\delta+k-1}$, with respective slopes $\l_1 = \dfrac{-1}{p^{r-\delta +1}}$ and $\l_k = \dfrac{-1}{e_k} $ for every $ 2 \le k \le \delta$.  Thus $R_{\l_k}(F)(y)$ is irreducible over $\F_{ \ph}$ as it is of degree $1$ for every $1 \le k \le \delta$. By Theorem \ref{ore}, $p\Z_K = \prod_{i = 1}^{N_p(m,s,b)} \aF_i \cdot  \aF$, where  $\aF$ is a non-zero ideal of $\Z_K$,  $\aF_i = \prod_{k=1}^{\delta } \pF_{ik}^{e_k} $ such that  for every $1\le k \le \delta$, $1 \le i \le N_p(m,s,b)$,  $\pF_{ik}$  is a  prime ideal of $\Z_K$ with  residue degree  $ f(\pF_{ik}/p) =  \deg(R_{\l_k}(F)(y)) \times m = m $. Therefore, the $N_p(m,s,b)$ monic  irreducible factors of $F(x)$ modulo $p$ provide $\omega \cdot N_p(m,s,b) $ prime ideals of $\Z_K$ above  $p$. By applying Lemma \ref{comindex}, if $\omega \cdot N_p(m,s,b)> N_p(m) $, then $p$ is a prime common index divisor of $K$. So, $K$ is not monogenic.\\
 Now, let $\ph(x)$ be a monic linear monic factors of $F(x)$ modulo a prime $p$; $\overline{\ph(x)}$ is a monic irreducible   factor of $x^s+\overline{b}$ in $\F_p[x]$. 
	\item  Since $\mu< \min(\nu,r+1) $, then  $\omega_0 = \mu$ and $\overline{\left(\dfrac{A_0(x)}{p^{\omega_0}}\right)}=\overline{\left(\dfrac{A_0(x)}{p^{\mu}}\right)} = a_p \bar{x}$. So $\overline{\ph(x)} \nmid \overline{\left(\dfrac{A_0(x)}{p^{\omega_0}}\right)}$ (because $0$ is not a root of $F(x)$ modulo $p$). It follows that the $\ph(x)$-development (\ref{devn1}) of  $F(x)$ is admissible. By Lemma \ref{admiss}, $\npp{F} = S_1+S_2+\cdots+S_{\mu}$ has $\mu$ sides of degree $1$ each joining the   points $\{(0,\mu)\} \cup \{(p^{r-k},k),\, k=0, \ldots, \mu-1 \}$.  Thus $R_{\l_k}(F)(y)$ is irreducible over $\F_{\phi}$ for every $ 1 \le k \le \mu$. By Theorem \ref{ore}, any linear monic irreducible factor of the polynomial $F(x)$ modulo $p$ provides $\mu$ prime ideals of $\Z_K$ lying  over $p$ of residue degree $1$ each. According to  Lemma \ref{comindex}, if  $ N_p(1) = p < \mu \cdot N_p(1,s,b)$, then $p \mid i(K)$. Hence $K$ is not monogenic. 
\item Since $\nu <\min(\mu,r+1)$, then $\omega_0 = \n$ and
$\overline{\left(\dfrac{A_0(x)}{p^{\omega_0}}\right)}=\overline{\left(\dfrac{A_0(x)}{p^{\n}}\right)} =((-b)^{p^r}+b)_p$. So $\overline{\ph(x)} \nmid  \overline{\left(\dfrac{A_0(x)}{p^{\omega_0}}\right)}$, then the $\ph(x)$-development (\ref{devn1}) of  $F(x)$ is admissible. By Lemma \ref{admiss}, $\npp{F} = S_1+S_2+\cdots+S_{\n}$ has $\n$ sides of degree $1$ each joining the points $(0, \n), (p^{r-\n+1}, \n -1), (p^{r-\n+2}, \n -2), \ldots, (p^r,0)$ in the Euclidean plane. Thus  $R_{\l_k}(F)(y)$ is irreducible over $\F_{\phi}$ for every $ 1 \le k \le \nu$. By Theorem \ref{ore}, any linear  monic irreducible factor of the polynomial $F(x)$ modulo $p$ provides $\n$ prime ideals of $\Z_K$ over $p$ of residue degree $1$ each.  It follows by Lemma \ref{comindex} that, if  $ N_p(1) = p < \n \cdot N_p(1,s,b)$, then $p $ is a prime common index divisor of $K$. So, $K$ is not monogenic.
\item Since $r+1 < \min (\mu,\nu)$, then
$\overline{\left(\dfrac{A_0(x)}{p^{\omega_0}}\right)}=\overline{\left(\dfrac{A_0(x)}{p^{r+1}}\right)} = \overline{H(x)}$. So $\overline{\ph(x)} \nmid  \overline{\left(\dfrac{A_0(x)}{p^{\omega_0}}\right)}$. Then the $\ph(x)$-development (\ref{devn1}) of  $F(x)$ is admissible. Thus $\npp{F} = S_0+S_2+\cdots+S_{r}$ has $r+1$ sides of degree $1$ each. By Theorem \ref{ore}, any linear monic irreducible factor of the polynomial $F(x)$ modulo $p$ provides $r+1$ prime ideals of $\Z_K$ over $p$ of residue degree $1$ each. Using Lemma \ref{comindex}, if  $N_p(1) = p < (r+1) \cdot N_p(1,s,b)$, then $p \mid i(K)$. Consequently, the field  $K$ is not monogenic.
\end{enumerate}	

\begin{figure}[htbp] 
	\centering
	\begin{tikzpicture}[x=0.16cm,y=0.45cm]
	\draw[latex-latex] (0,6.8) -- (0,0) -- (83,0) ;
	\draw[thick] (0,0) -- (-0.5,0);
	\draw[thick] (0,0) -- (0,-0.5); 
	\draw[thick,red] (1,-2pt) -- (1,2pt);
	\draw[thick,red] (3,-2pt) -- (3,2pt);
	\draw[thick,red] (9,-2pt) -- (9,2pt);
	\draw[thick,red] (27,-2pt) -- (27,2pt);
	\draw[thick,red] (81,-2pt) -- (81,2pt);
	\draw[thick,red] (-2pt,1) -- (2pt,1);
	\draw[thick,red] (-2pt,2) -- (2pt,2);
	\draw[thick,red] (-2pt,3) -- (2pt,3);
	\draw[thick,red] (-2pt,4) -- (2pt,4);
	\node at (0,0) [below left,blue]{\footnotesize  $0$};
	\node at (1,0) [below ,blue]{\footnotesize  $1$};
	\node at (3,0) [below ,blue]{\footnotesize $3$};
	\node at (9,0) [below ,blue]{\footnotesize  $3^{2}$};
	\node at (27,0) [below ,blue]{\footnotesize  $3^{3}$};
	\node at (81,0) [below ,blue]{\footnotesize  $3^{4}$};
	\node at (0,1) [left ,blue]{\footnotesize  $1$};
	\node at (0,2) [left ,blue]{\footnotesize  $2$};
	\node at (0,3) [left ,blue]{\footnotesize  $3$};
	\node at (0,4) [left ,blue]{\footnotesize  $4$};
	\draw[thick,mark=*] plot coordinates{(0,6) (1,4)(3,3) (9,2) (27,1) (81,0)};
	\node at (0.3,4.2) [above right  ,blue]{\footnotesize  $S_{1}$};
	\node at (1.8,3.1) [above right  ,blue]{\footnotesize  $S_{2}$};
	\node at (6,2.1) [above right  ,blue]{\footnotesize  $S_{3}$};
	\node at (16,1.25) [above right  ,blue]{\footnotesize  $S_{4}$};
	\node at (50,0.3) [above right  ,blue]{\footnotesize  $S_{5}$};
	\end{tikzpicture}
	\caption{    \large  $\npp{F}$ at $p=3$, where $r=4$ and $\delta > 5$.\hspace{5cm}}
\end{figure}
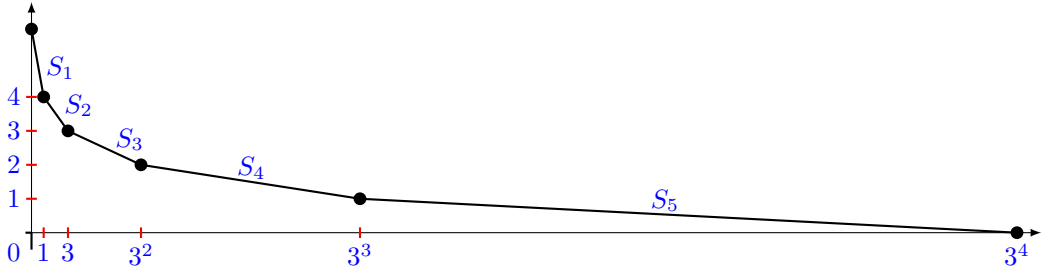
\end{proof}

\begin{proof}[Proof of Theorem \ref{dn2}]\	
\\\begin{enumerate}
	\item 	By hypothesis  $p\mid b$, $p\mid n-1$, and $p\nmid a$, $F(x) \equiv x(x^u+a)^{p^k} \md p$. Let $\ph(x) \in \Z[x]$ be a monic polynomial  of degree $m>1$  such that  $\overline{\ph(x)}$ is an  irreducible factor of the polynomial $F(x)$ modulo $p$:  $\overline{\ph(x)}$  is a monic  irreducible factor of  $x^u+\overline{a }$ in $\F_p[x]$. By using  Lemma \ref{lemtech}, we set  $x^u +a = \ph(x) U(x) + pT(x)$, where $U(x), T(x) \in \Z[x]$ such that $\overline{\ph(x)} \nmid \overline{U(x) T(x)}$. Write 
	\begin{eqnarray*}
	F(x)&=& x(x^u)^{p^k} + ax +b = x(x^u +a -a)^{p^k}+ax+b \nonumber\\
	&=&x(\ph(x) U(x) + p T(x)-a)^{p^k}+ax+b\nonumber \\
	\end{eqnarray*}
Applying Binomial theorem, we see that 
\begin{eqnarray*}
F(x)= x(\ph(x) U(x))^{p^k} +  \sum_{j=1}^{p^k-1} \binom{p^k}{j}x N(x)^{p^k-j} U(x)^j \ph(x)^j+ xN(x)^{p^k}+ax+b
\end{eqnarray*}

where $N(x) = p T(x)-a$. By Lemma \ref{binomial},  $N(x)^{p^k} = p^{k+1}H(x)+(-a)^{p^k}$, where $$H(x)= a^{p^k-1}T(x)+\frac{1}{p^{k+1}} \sum_{j=0}^{p^k-2} (-1)^j \binom{p^k}{j}a^j(pT(x))^{p^k-j}.$$It follows that 
\begin{eqnarray}\label{dn12}
F(x)=x (\ph(x) U(x))^{p^k} +  \sum_{j=1}^{p^k-1} \binom{p^k}{j} xN(x)^{p^k-j} U(x)^j \ph(x)^j+ p^{k+1}xH(x)+((-a)^{p^k}+a)x+b.
\end{eqnarray}	
Thus $F(x)= \sum_{j=0}^{p^k} A_j(x)\ph(x)^j$, where 	
	
	\[  \begin{cases}
	A_0(x) = p^{k+1} x H(x)+((-a)^{p^k}+a)x+b,\,\\
	A_j(x) =\displaystyle\binom{p^k}{j} x N(x)^{p^k-j}U(x)^j\,\, \mbox{for\, every } 1 \leq j \leq p^k.
	\end{cases} \]Note that $\overline{N(x)} = \bar{a}$. By using Lemma \ref{binomial}, we see that  $\omega_j = \nu_p(A_j(x))= \n_p(\dbinom{p^k}{j}) =k-\n_p(j)$, then $\overline{\left(\dfrac{A_j(x)}{p^{\omega_j}}\right)}= \dbinom{p^k}{j}_p  \overline{ x N(x)^{p^k-j}U(x)^j}$. Thus $\overline{\ph(x)} \nmid \overline{\left(\dfrac{A_j(x)}{p^{\omega_j}}\right)}$ for every $1\le j \le p^k$. Moreover, If $\t > k+1$, then $\omega_0=\n_p(A_0(x))  =k+1 = \k$ and  $\overline{\left(\dfrac{A_0(x)}{p^{\omega_0}}\right)}= \overline{xH(x)} = \overline{a^{p^k-1}  x T(x)}$. So $\overline{\ph(x)} \nmid \overline{\left(\dfrac{A_0(x)}{p^{\omega_0}}\right)}$. It follows that the $\ph$-development (\ref{dn12}) of $F(x)$ is admissible. By Lemma \ref{admiss}, $\npp{F}= S_0+S_1+\cdots +S_{k}$ has $ k+1$ sides of degree $1$ each with respective ramification indices $e_0=1$, $e_i = (p-1)p^{i-1}$, with respective slopes $\l_1 = -1$, $\l_i = \frac{-1}{e_i}$ for every $1 \le i \le k$. Thus  $R_{\l_i}(F)(y)$ is irreducible over $\F_{ \ph}$  as it is of degree $1$, $0 \le i \le k$. It follows that $F(x)$ is $\ph(x)$-regular with respect to $\n_p$. By Theorem \ref{ore},  $p\Z_K = \prod_{t = 1}^{N_p(m,u,a)} \aF_t \cdot  \aF$, where  $\aF$ is a non-zero ideal of $\Z_K$,  $\aF_t = \prod_{i=0}^{k } \pF_{ti}^{e_i} $ such that  for every $1 \le t \le N_p(m,u,a)$,  $0\le i \le k$,  $\pF_{ti}$  is a  prime ideal of $\Z_K$ with  residue degree  $ f(\pF_{ti}/p) = m \times \deg(R_{\l_i}(F)(y)) = m $. Thus the $N_p(m,u,a)$ monic  irreducible factors of $F(x)$ modulo $p$ provide $\k \cdot N_p(m,u,a) $ prime ideals of $\Z_K$ over $p$ of residue degree $m$ each. By Lemma \ref{comindex}, if  $\k \cdot N_p(m,u,a)> N_p(m) $, then $p\mid i(K)$. Consequently, $K$ is not monogenic.
Similarly, if $k+1 > \t$, then $\omega_0 = \t= \k$ and $\overline{\left(\dfrac{A_0(x)}{p^{\omega_0}}\right)} = \overline{\left(\dfrac{A_0(x)}{p^{\t}}\right)}= ((-a)^{p^k}+a)_p \bar{x}+b_p \neq \bar{0}$. Since $\deg(\ph(x)) = m >1$, then $\overline{\ph(x)}$ does not divide $\overline{\left(\dfrac{A_0(x)}{p^{\omega_0}}\right)}$. It follows that the $\ph$-development (\ref{dn12}) of $F(x)$ is admissible. By Lemma \ref{admiss}, $\npp{F}= S_1+S_2+\cdots +S_{\t}$ has $\t$ sides of degree $1$ each with respective ramification indices $e_1= p^{k+1-\t}$, $e_i = (p-1)\cdot  p^{k - \t + i - 1 }$ for every $2 \le i \le \t$ and with respective slopes $\l_i = \frac{-1}{e_i}$ for every $1 \le i \le \t$.  Thus $R_{\l_i}(F)(y)$ is irreducible over $\F_{ \ph}$ for every $1 \le i \le \t$. By Theorem \ref{ore}, the monic irreducible factor $\ph(x)$ of $F(x)$ modulo $p$ provides $ \t$ prime ideals above $p$ with the same residue degree $f = \deg(\ph) \times \deg(R_{\l_i}(F)(y)) = m \times 1 = m $, $1 \le i \le \k$. It follows that the $N_p(m,u,a)$ monic  irreducible factors  of the polynomial $F(x)$ provide $\k \cdot N_p(m,u,a)$, prime ideal of $\Z_K$ lying  above $p$ of residue degree $m$ each. By Lemma \ref{comindex}, if $\k \cdot N_p(m,u,a) > N_p(m)$, then $p \mid i(K)$. So, $K$ is not monogenic.	\\
\item For the proof of Theorem \ref{dn2}(2), (3) and (4), we will use the linear monic irreducible factors of $\overline{F(x)}$ in  $\F_p[x]$ to prove that $i(K)>1$, namely   $x$ and the linear  monic factors of $x^u +\overline{a}$.  Since $\n_{x}(\overline{F(x)})=1$, then the factor   $x$ provide a unique prime ideal over $p$ of residue degree $1$. For the linear monic factors  of $x^u + \overline{a}$, we proceed by analogous to the proof of Theorem \ref{dn12} (1), (2) and (3); by using   Lemma \ref{binomial}, we get the exact  value of $\omega_0 = \n_p(A_0(x))$, we verify that  $\overline{\ph(x)}$ does not divide $\overline{\left(\dfrac{A_0(x)}{p^{\omega_0}}\right)}$ which ensures that the $\ph$-development (\ref{dn12}) of $F(x)$ is  admissible. After determining Newton polygons, we prove that  $F(x)$ is $\ph$-regular. Finally, we apply Theorem \ref{ore} and Lemma \ref{comindex} to show that $p\mid i(K)$.
\end{enumerate}
\end{proof}
\begin{proof}[Proof of Corollaries \ref{corn11} and \ref{corn12} ]\
	\\According to the Explicit factorization  of the polynomial  $x^{2^k}+1$ into product of monic irreducible polynomials  over $\F_p$ with prime $p \equiv 3 \md 4$ given in  \cite[Theorem 1 and Corollary 3]{Blake}, and after calculations we have the following complete factorization of the polynomials  $x^{2^k} - 1$ into product of irreducible polynomials in  $ \F_p[x]$: for every rational positive integer   $k\ge 3$, we have $$ x^{2^k} - 1 = (x-1)(x-2)(x^2-2)(x^2-x-1)(x^2-2x-1) \prod_{ \underset{1 \le t \le k-3\,\,/ k \,\,\ge 4}{u\in \{\frac{-1}{2}, \frac{1}{2}\}} }(x^{2^{t+1}}-2ux^{2^t}-1)\md 3. $$It follows that $N_3(1, 2^k,2)=2$ for every $k \ge 1$,   $N_3(2, 2^2,2)=1$  and $N_3(2, 2^k, 2) = 3$ for every $k \ge 3$. On the other hand $N_3(2, 2^1, 1)=1$ and   $N_3(2, 2^2, 1)=2$.  Then by a direct  applications of   Theorem \ref{dn1} and \ref{dn2}, we conclude  the two corollaries.
	
\end{proof}
\begin{proof}[Proof of Theorem \ref{d51}]\
	\\ First, we note that in Theorem \ref{d51}(1),\dots, (6), we have
 $2 \mid b$ and   $2\nmid a$. Then $F(x)\equiv x(x-1)^4 \md 2$. Set $\ph = x-1$. The $\ph$-adic development of $F(x)$ is   
 \begin{eqnarray}\label{devd51}
F(x) = \ph^5 + 5 \ph^4+10 \ph^3+10 \ph^2+(5+a)\ph+(1+a+b).
 \end{eqnarray}

\begin{enumerate}
 	\item	Since $a\equiv 1 \md 4$ and $b\equiv 2  \md 4$, then $\n_2(a+5)=1$ and $\n_2(1+a+b) \ge 2$. By the $\ph$-adic development (\ref{devd51}) of $F(x)$,
	 $\npp{F} = S_1+S_2$ has two sides of degree $1$ each, with respective ramification indices $e_1=1$ and $e_2=3$  (see FIGURE $4$). Thus  $R_{\l_k}(F)(y)$ is irreducible over $\F_{ \ph}$, $k = 1, 2$. By Theorem \ref{ore}, $2\Z_K= \pF_0 \pF_1\pF_2^3$, with residue degrees $f(\pF_k/2)=1$ for  $k=0,1,2$. It follows by Lemma \ref{comindex} that $2\mid i(K)$, and so $K$ is not monogenic.

	\begin{figure}[htbp] 
		\begin{tikzpicture}[x=1.5cm,y=0.5cm]
		\draw[latex-latex] (0,4) -- (0,0) -- (4.5,0) ;
		\draw[thick] (0,0) -- (-0.5,0);
		\draw[thick] (0,0) -- (0,-0.5); 
		\draw[thick,red] (1,-2pt) -- (1,2pt);
		\draw[thick,red] (3,-2pt) -- (3,2pt);
		\draw[thick,red] (2,-2pt) -- (2,2pt);
		\draw[thick,red] (4,-2pt) -- (4,2pt);
		\draw[thick,red] (-2pt,1) -- (2pt,1);
		\draw[thick,red] (-2pt,2) -- (2pt,2);
		\draw[thick,red] (-2pt,3) -- (2pt,3);
		
		\node at (0,0) [below left,blue]{\footnotesize  $0$};
		\node at (1,0) [below ,blue]{\footnotesize  $1$};
		\node at (2,0) [below ,blue]{\footnotesize $2$};
		\node at (3,0) [below ,blue]{\footnotesize  $3$};
		\node at (4,0) [below ,blue]{\footnotesize  $4$};
		\node at (0,1) [left ,blue]{\footnotesize  $1$};
		\node at (0,2) [left ,blue]{\footnotesize  $2$};
			\draw[thick, only marks, mark=x] plot coordinates{ (2,1)  (3,1) };
		\draw[thick,mark=*] plot coordinates{(0,3) (1,1) (4,0)};
		\node at (0.5,1.5) [above right  ,blue]{\footnotesize  $S_{1}$};
		\node at (2.5,0.2) [above right  ,blue]{\footnotesize  $S_{2}$};

		\end{tikzpicture}
		\caption{   \hspace{5cm}}
	\end{figure}

	\item Since  $(\bar{a},\bar{b})= ( \bar{7}, \bar{8}) \,\,\mbox{or}\,\, (\bar{15}, \bar{0})$ in $(\Z/{16\Z})^2$, $\n_2(a+5)=2$ and $\n_2(1+a+b) \ge 4$. By the $\ph$-adic development (\ref{devd51}) of $F(x)$, $\npp{F}=S_1+S_2+S_3$ has three sides with the same degree $1$, with respective ramification indices $e_1 = e_2 =1 $ and $e_3=2$ (see FIGURE $5$). Thus  $R_{\l_k}(F)(y)$ is irreducible over $\F_{ \ph}$, $k = 1, 2, 3$. By Theorem \ref{ore}, $2\Z_K= \pF_0\pF_1\pF_2\pF_3^2 $, where $\pF_k$ is a prime ideal of $\Z_K$ with residue degree $f(\pF_k/2) =1$, $k=0,1,2,3$. By Lemma \ref{comindex}, $2\mid i(K)$. Hence $K$ is not monogenic. 
	
	\begin{figure}[htbp] 
		\begin{tikzpicture}[x=1.5cm,y=0.3cm]
		\draw[latex-latex] (0,6) -- (0,0) -- (4.5,0) ;
		\draw[thick] (0,0) -- (-0.5,0);
		\draw[thick] (0,0) -- (0,-0.5); 
		\draw[thick,red] (1,-2pt) -- (1,2pt);
		\draw[thick,red] (3,-2pt) -- (3,2pt);
		\draw[thick,red] (2,-2pt) -- (2,2pt);
		\draw[thick,red] (4,-2pt) -- (4,2pt);
		\draw[thick,red] (-2pt,1) -- (2pt,1);
		\draw[thick,red] (-2pt,2) -- (2pt,2);
		\draw[thick,red] (-2pt,3) -- (2pt,3);
		\draw[thick,red] (-2pt,4) -- (2pt,4);
		
		\node at (0,0) [below left,blue]{\footnotesize  $0$};
		\node at (1,0) [below ,blue]{\footnotesize  $1$};
		\node at (2,0) [below ,blue]{\footnotesize $2$};
		\node at (3,0) [below ,blue]{\footnotesize  $3$};
		\node at (4,0) [below ,blue]{\footnotesize  $4$};
		\node at (0,1) [left ,blue]{\footnotesize  $1$};
		\node at (0,2) [left ,blue]{\footnotesize  $2$};
		\node at (0,3) [left ,blue]{\footnotesize  $3$};
		\node at (0,4) [left ,blue]{\footnotesize  $4$};
			\draw[thick, only marks, mark=x] plot coordinates{  (3,1) };
		\draw[thick,mark=*] plot coordinates{(0,5) (1,2)(2,1) (4,0)};
		\node at (0.5,3) [above right  ,blue]{\footnotesize  $S_{1}$};
		\node at (1.5,1.1) [above right  ,blue]{\footnotesize  $S_{2}$};
		\node at (3.2,0.1) [above right  ,blue]{\footnotesize  $S_{3}$};
		
		\end{tikzpicture}
		\caption{   \hspace{5cm}}
	\end{figure}

\item Since  $(\bar{a},\bar{b}) = (\overline{19},\overline{4}) \, \mbox{or}\,\,(\overline{3},\overline{20})$ in  $(\Z/{32\Z})^2$,  $\n_2(a+5)=3$ and $\n_2(1+a+b) =3 $. Thus   $\npp{F} = S_1+S_2$ has two sides with respective  degrees $2$ and $1$. But, the residual polynomial  attached te the segment $S_1$: $R_{\l_1}(F)(y) = 1+y^2 =(1+y)^2 $ is not separable over $\F_{{\ph}} \simeq \F_2$. Thus Theorem \ref{ore} is not applicable. Replace $\ph(x)$  by $\psi(x) := x-3$. The $\psi$-adic developement of $F(x)$ is

\begin{eqnarray}\label{devpsi51}
F(x) = \psi^5+ 15\psi^4+90 \psi^3+ 270 \psi^2+(a+5+400)\psi+ 1+a+b+2(a+5)+232.
\end{eqnarray}Moreover, we have $$\n_2(a+5+400)=3\,\, \mbox{and}\,\, \n_2(1+a+b+2(a+5)+232)=4.$$ Thus  $N_{\psi}^+(F) = S_1^{'}+S_2^{'}$ has two sides of degree $1$ and ramifications index $2$ each (see FIGURE $6$).  Thus  $R_{\l^{'}_k}(F)(y)$ is irreducible over $\F_{ \psi} \simeq \F_2$, $k = 1, 2$. By Theorem \ref{ore}, $2 \Z_K = \pF_0 \pF_1^2 \pF_2^2 $, with residue degrees $f(\pF_k/2) =1$, $k=0,1,2$. By Lemma \ref{comindex}, $2 \mid i(K)$. Hence $K$ is not monogenic.
\begin{figure}[htbp] 
	\centering
\begin{tikzpicture}[x=1.5cm,y=0.4cm]
\draw[latex-latex] (0,5) -- (0,0) -- (4.5,0) ;
	\draw[thick] (0,0) -- (-0.5,0);
	\draw[thick] (0,0) -- (0,-0.5); 
	\draw[thick,red] (1,-2pt) -- (1,2pt);
	\draw[thick,red] (3,-2pt) -- (3,2pt);
	\draw[thick,red] (2,-2pt) -- (2,2pt);
	\draw[thick,red] (4,-2pt) -- (4,2pt);
	\draw[thick,red] (-2pt,1) -- (2pt,1);
	\draw[thick,red] (-2pt,2) -- (2pt,2);
	\draw[thick,red] (-2pt,3) -- (2pt,3);
	
	\node at (0,0) [below left,blue]{\footnotesize  $0$};
	\node at (1,0) [below ,blue]{\footnotesize  $1$};
	\node at (2,0) [below ,blue]{\footnotesize $2$};
	\node at (3,0) [below ,blue]{\footnotesize  $3$};
	\node at (4,0) [below ,blue]{\footnotesize $4$};
	\node at (0,3) [left ,blue]{\footnotesize  $3$};
	\node at (0,4) [left ,blue]{\footnotesize  $4$};
	\node at (0,1) [left ,blue]{\footnotesize  $1$};
	\node at (0,2) [left ,blue]{\footnotesize  $2$};
		\draw[thick, only marks, mark=x] plot coordinates{ (1,3)  (3,1) };
	\draw[thick,mark=*] plot coordinates{(0,4) (2,1) (4,0)};
	\node at (1,1.8) [above right  ,blue]{\footnotesize  $S_{1}^{'}$};

	\node at (2.5,0.3) [above right  ,blue]{\footnotesize  $S_{2}^{'}$};
	
	\end{tikzpicture}
	\caption{   \hspace{5cm}}
\end{figure}

\item Since $(\overline{a},\overline{b})= ( \overline{3}, \overline{4}),  (\overline{35}, \overline{36}), ( \overline{19}, \overline{20}) \,\, \mbox{or}\,\,  (\overline{51}, \overline{52})$ in $(\Z/{64\Z})^2$, we have$$\n_2(a+5+400)=3\,\, \mbox{and}\,\, \n_2(1+a+b+2(a+5)+232) \ge 6.$$  According to  the $\psi$-adic development (\ref{devpsi51}) of $F(x)$,  we conclude that
 $N_{\psi}^+(F) = S_1^{'}+S_2^{'}+S_3^{'}$ has three sides of degree $1$ each  with respective ramification indices $e_1^{'}=e_2^{'}=1$ and $e_3^{'}=2$ (see FIGURE $7$). Thus by Theorem \ref{ore}, $2\Z_K = \pF_0 \pF_1 \pF_2 \pF_3^2$, with residue degrees $f(\pF_k/2)=1$ for every $0 \le k \le 3$. By Lemma \ref{comindex}, $2\mid i(K)$, and so $K$ is not monogenic.

	\begin{figure}[htbp] 
		\centering
	\begin{tikzpicture}[x=1.3cm,y=0.3cm]
	\draw[latex-latex] (0,7) -- (0,0) -- (4.5,0) ;
	\draw[thick] (0,0) -- (-0.5,0);
	\draw[thick] (0,0) -- (0,-0.5); 
	\draw[thick,red] (1,-2pt) -- (1,2pt);
	\draw[thick,red] (3,-2pt) -- (3,2pt);
	\draw[thick,red] (2,-2pt) -- (2,2pt);
	\draw[thick,red] (4,-2pt) -- (4,2pt);
	\draw[thick,red] (-2pt,1) -- (2pt,1);
	\draw[thick,red] (-2pt,2) -- (2pt,2);
	\draw[thick,red] (-2pt,3) -- (2pt,3);
		\draw[thick,red] (-2pt,4) -- (2pt,4);
	\draw[thick,red] (-2pt,5) -- (2pt,5);
	\node at (0,0) [below left,blue]{\footnotesize  $0$};
	\node at (1,0) [below ,blue]{\footnotesize  $1$};
	\node at (2,0) [below ,blue]{\footnotesize $2$};
	\node at (3,0) [below ,blue]{\footnotesize  $3$};
	\node at (4,0) [below ,blue]{\footnotesize  $4$};
	\node at (0,1) [left ,blue]{\footnotesize  $1$};
	\node at (0,2) [left ,blue]{\footnotesize  $2$};
	\node at (0,3) [left ,blue]{\footnotesize  $3$};
	\node at (0,4) [left ,blue]{\footnotesize  $4$};
		\node at (0,5) [left ,blue]{\footnotesize  $5$};

	\draw[thick, only marks, mark=x] plot coordinates{  (3,1) };
	\draw[thick,mark=*] plot coordinates{(0,6) (1,3) (2,1) (4,0)};
	\node at (0.5,4) [above right  ,blue]{\footnotesize  $S_{1}^{'}$};
	\node at (1.5,1.4) [above right  ,blue]{\footnotesize  $S_{2}^{'}$};
		\node at (2.5,0.3) [above right  ,blue]{\footnotesize  $S_{3}^{'}$};
	
	\end{tikzpicture}
	\caption{   \hspace{5cm}}
\end{figure}

 \item  Since $(\bar{a},\bar{b}) = (\overline{3},\overline{12}) \, \mbox{or}\,\,(\overline{19},\overline{28})$ in  $(\Z/{32\Z})^2$,  $\n_2(a+5)=3$ and $\n_2(1+a+b) =4 $. Thus   $\npp{F} = S_1+S_2$ has two sides of degree $1$ and ramification index $2$ each (see FIGURE $8$). By Theorem \ref{ore}, $2\Z_K = \pF_1 \pF_2^2 \pF_3^2$ such that  $f(\pF_k/2)=2$, $ k = 0, 1, 2$. By Lemma \ref{comindex}, $2 \mid i(K)$ and so $K$ is not monogenic.  
  	\begin{figure}[htbp] 
	\begin{tikzpicture}[x=1.5cm,y=0.5cm]
	\draw[latex-latex] (0,5) -- (0,0) -- (4.5,0) ;
	\draw[thick] (0,0) -- (-0.5,0);
	\draw[thick] (0,0) -- (0,-0.5); 
	\draw[thick,red] (1,-2pt) -- (1,2pt);
	\draw[thick,red] (3,-2pt) -- (3,2pt);
	\draw[thick,red] (2,-2pt) -- (2,2pt);
	\draw[thick,red] (4,-2pt) -- (4,2pt);
	\draw[thick,red] (-2pt,1) -- (2pt,1);
	\draw[thick,red] (-2pt,2) -- (2pt,2);
	\draw[thick,red] (-2pt,3) -- (2pt,3);
	
	\node at (0,0) [below left,blue]{\footnotesize  $0$};
	\node at (1,0) [below ,blue]{\footnotesize  $1$};
	\node at (2,0) [below ,blue]{\footnotesize $2$};
	\node at (3,0) [below ,blue]{\footnotesize  $3$};
	\node at (4,0) [below ,blue]{\footnotesize $4$};
	\node at (0,3) [left ,blue]{\footnotesize  $3$};
	\node at (0,4) [left ,blue]{\footnotesize  $4$};
	\node at (0,1) [left ,blue]{\footnotesize  $1$};
	\node at (0,2) [left ,blue]{\footnotesize  $2$};
	\draw[thick, only marks, mark=x] plot coordinates{ (1,3)  (3,1) };
	\draw[thick,mark=*] plot coordinates{(0,4) (2,1) (4,0)};
	\node at (1,2) [above right  ,blue]{\footnotesize  $S_1$};
	
	\node at (2.5,0.5) [above right  ,blue]{\footnotesize  $S_2$};
	
	\end{tikzpicture}
	\caption{   \hspace{5cm}}
\end{figure}
\item Since  $(\overline{a},\overline{b})= ( \overline{3}, \overline{60}),  (\overline{19}, \overline{44}), ( \overline{35}, \overline{28}) \,\, \mbox{or}\,\,  (\overline{51}, \overline{12})$ in $(\Z/{64\Z})^2$, we have  $\n_2(a+5)=3$ and $\n_2(1+a+b) \ge 6 $. Thus   $\npp{F} = S_1+S_2+S_3$ has three sides with degree $1$ each and ramification indices $e_1 = e_2=1$ and $e_3=2$. By Theorem \ref{ore}, $2\Z_K =  \pF_0 \pF_1 \pF_2 \pF_3^2$, with $f(\pF_k/2) = 1$ for $k=0,1,2,3$. By Lemma \ref{comindex}, $2 \mid i(K)$. Consequently $K$ is not monogenic.

\item Since  $a \equiv 4 \md 8$ and $b \equiv 0 \md{8} $,  then $F(x) \equiv \ph^5 \md 2$, where $\ph = x$. It follows that $\npp{F} = S_1+S_2$ has two sides with respective degrees  $d(S_1)=1$, $d(S_2)=2$ and ramification indices $e_1 = 1$, $e_2 = 2$. Their attached residual polynomials $R_{\l_1}(F)(y)=1+y$ and $R_{\l_2}(F)(y)=(1+y)^2$. Then $2\Z_K =\pF_1 \aF^2$ where $\pF_1$ is a prime ideal with $f(\pF_1/2)=1$ and $\aF$ is proper ideal of $\Z_K$. Let us use second order Newton polygon as introduced in \cite{Nar}, let $\ph_2 = x^2 - 2$ ans $\n_2^{(2)}$ be the valuation of second order induced by $\ph_2$. The $\ph_2$-adic developement of $F(x)$ is given by: 
\begin{eqnarray*}
	F(x)=x\ph_2^2 + 4x\ph_2+(4+a)x+b.
\end{eqnarray*}
 By \cite[Theorem 2.11 and Proposition 2.7]{Nar}, we get \begin{eqnarray*}
\hspace{1.3cm} \n_2^{(2)}(x)=1,\, \n_2^{(2)}(\ph_2)=2,\, \n_2^{(2)}(x\ph_2^2)=5,\, \n_2^{(2)}(4x\ph_2)=7,\, \n_2^{(2)}((4+a)x+b)\ge 10.
\end{eqnarray*}
Hence the $\ph_2-$Newton polygon of second order $N_2(F) = S_1^{(2)}+ S_2^{(2)}$ has two sides of degree one each  (see FIGURE $9$), with respective slopes $\l^2_1  \le -3$ and $\l^2_2 = -2$. Their attached residual polynomials  $R^{(2)}_{\l^2_1} (F)(y)=R^{(2)}_{\l^2_2 } (F)(y)=1+y$ in $\F^2[y]$. Thus by  \cite[Theorem 3.1 and 3.4]{Nar} in second order,   $2 \Z_K = \pF_1 (\pF_2 \pF_3)^2$, where $f(\pF_k/2)= 1$ for every $k = 1,2,3$. It follows that for $p=2$, $3=P_1>N_2(1)=2$. Then, by  Lemme \ref{comindex}, $2 \mid i(K)$. So $K$ is not monogenic.
\begin{figure}[htbp] 
	\begin{tikzpicture}[x=3cm,y=0.3cm]
	\draw[latex-latex] (0,9) -- (0,0) -- (2.5,0) ;
	\draw[thick] (0,0) -- (-0.5,0);
	\draw[thick] (0,0) -- (0,-0.5); 
	\draw[thick,black] (1,-2pt) -- (1,2pt);
	\draw[thick,black] (2,-2pt) -- (2,2pt);
	\draw[thick,black] (-2pt,1) -- (2pt,1);
	\draw[thick,black] (-2pt,6) -- (2pt,6);
	\draw[thick,black] (-2pt,3) -- (2pt,3);
	\node at (0,0) [below left,black]{\footnotesize  $0$};
	\node at (1,0) [below , black]{\footnotesize  $1$};
	\node at (2,0) [below ,black]{\footnotesize $2$};

	\node at (0,1) [left ,black]{\footnotesize  $5$};
	\node at (0,3) [left ,black]{\footnotesize  $7$};
	\node at (0,6) [left ,black]{\footnotesize  $10$};

	\draw[thick,mark=*] plot coordinates{(0,7) (1,3)(2,1) };
	\draw[thick, dashed] plot coordinates{(0,1) (2.4,1) };
		\draw[thick, dashed] plot coordinates{(0,0) (0,1) };
	\node at (0.5,4.4) [above right  ,black]{\footnotesize  $S_1^2$};
	\node at (1.4,1.7) [above right  ,black]{\footnotesize  $S_2^2$};

	\end{tikzpicture}
	\caption{      \hspace{5cm}}
\end{figure}

\end{enumerate}

\end{proof}

\begin{proof}[Proof of Theorem \ref{d61}]\
	\\ First, we note  that in   Theorem \ref{d61}(1),(2) and (3)
we have $2 \mid a$ and $2\nmid b$, then $ F(x) \equiv (\ph \cdot  \psi)^2 \md 2$, where $\ph = x-1$ and $\psi = x^2+x+1$. The $\ph$-adic development of $F(x)$ is 
\begin{eqnarray}\label{dev61phi}
F(x)= \ph^6+6 \ph^5+15 \ph^4+20 \ph^3 + 15 \ph^2+(6+a)\ph+1+a+b,
\end{eqnarray}
and the $\psi$-adic development of $F(x)$ is\begin{eqnarray} \label{dev61psi}
F(x)=\psi^3-3x\psi^2+(2x-2)\psi+ ax+1+b . 
\end{eqnarray}
\begin{enumerate}
	\item Since $a \equiv 0 \md 8$ and $b \equiv 7 \md 8$, then $ \n_2(ax+1+b)= \min(\n_2(a), \n_2(1+b)) \ge 3$. According to the $\psi$-adic development (\ref{dev61psi}) of $F(x)$, then  $\nppsi{F} = S_1+S_2$ has two sides of degree $1$ each. Thus $R_{\l_k}(F)(y)$ is irreducible over $\F_{ \psi}$, $k = 1,2$. By applying Theorem \ref{ore}, the irreducible factor $\overline{\psi}$ of $\overline{F(x)}$ provides two  prime ideals of $\Z_K$  lying  over the rational prime $2$ with the same residue degree $f= deg(\psi)\cdot R_{\l_k}(F)(y)=2 \times 1=2$. Using Lemma \ref{comindex}, we see that $2 \mid i(K)$. Consequently $K$ is not monogenic. 
	\item Since $a=2 \md 4$, $b \equiv 1 \md 4$, then  $\n_2(ax+1+b)=1$. By  the $\psi$-adic development (\ref{dev61psi}) of $F(x)$, $\nppsi{F}=T$ is only one side of  degree $1$ joining the points $(0,1)$ and $(2,0)$, with ramification index $2$ and with slope $\l = \frac{-1}{2}$, thus $R_{\l}(F)(y)$ is irreducible over $\F_{ \psi}$. On the other hand, as $\n_2(1+a+b) = 2 \n_2(a+6)$. Then,  according to (\ref{dev61phi}), $\npp{F}=S$ is only one side of degree $2$, with ramification index $1$ and with slope $\l^{'} \le -2$. Moreover  $R_{\l^{'}}(F)(y)= 1+y+y^2$ which is irreducible over $\F_{ \ph} \simeq \F_2$.  Applying Theorem \ref{ore}, we see that $2 \Z_K = \pF_1 \pF_2^2$, with $f(\pF_1/2)= f(\pF_2/2)=2$. Thus there are two prime ideals of $\Z_K$ of degree residue degree $2$ each lying above the rational prime $2$. It follows that for the prime $p=2$, we have  $2=P_2>1=N_2(2)$. By Lemma \ref{comindex},  $2 \mid i(K)$ and so $K$ is not monogenic.  The particular case correspond to $\n_2(a+6)=2$ and $\n_2(1+a+b)=4$.
	
	\item Since  $a \equiv 0 \md 8$ and $b \equiv 3 \md 8$, then $\n_2(ax +1+b)=2$. By (\ref{dev61psi}),  $\nppsi{F} = S$ is only one side of degree $2$ joining the points $(0,2)$, $(1,1)$ and $(2,0)$  with ramification index $1$ and slop $\l = -1$. Moreover, $R_{\l}(F)(y)= \a y^2+ (1+\a)y+1   \in \F_{ \psi}[y]$, where $\a \in  \overline{\F_2}$ such that $\a^2+ \a+1= 0$. Then
	$R_{\l}(F)(y) =\a(y-1)(y-\a^2)$, which is separable in $\F_{\psi}[y]$. Also, by (\ref{dev61phi}),  $\npp{F}=T$ is only one side of degree $2$ with ramification index $1$ joining the points $(0,2)$ and $(0,2)$, with slope $\l^{'}= -1$ and the associated  residual polynomial is $R_{\l^{'}}(F)(y)= 1 +y + y^2$ which is irreducible over $\F_{ \ph} \simeq \F_2$.    By Theorem \ref{ore}, $2\Z_K = \pF_1 \pF_2  \pF_3$ with ramification indices $f(\pF_1/2)= f(\pF_2/2)= f(\pF_3/2)=2$. Then for $p=2$, we have $3=P_2> N_2(2)=1$. By Lemma \ref{comindex}, $2$ is a prime common index divisor of $K$. So, the field $K$ is not monogenic.

	\item Since  $a\equiv 0 \md 9$ and $b \equiv -1 \md 9$, then $F(x) \equiv  (\varphi \cdot \chi)^3 \md 3$, where $\varphi = x-1$ and $\chi = x-2$.  The $\varphi$-adic development of $F(x)$ is 
	 \begin{eqnarray}\label{devphi61b2}
F(x)= \varphi^6+6 \varphi^5+15 \varphi^4+20 \varphi^3 + 15 \varphi^2+(6+a)\varphi+1+a+b \nonumber
	 \end{eqnarray}
and the $\chi$-adic development of $F(x)$	is 
\begin{eqnarray}
F(x)= \chi^6+ 12 \chi^5+ 60 \chi^4+ 160 \chi^3+ 240 \chi^2 +(192+a)\chi+ 2a+b+64 \nonumber
\end{eqnarray}
We have also
$$ \n_3(6+a)= \n_3(192+a) = 1,\, \n_3(1+a+b) \ge 2\, \mbox{and }\, \n_3(2a+b+64)  \ge 2.   $$
 According to the above $\varphi$ and $\chi$-adic developments of the polynomial $F(x)$, both $N_{\varphi}^{+}(F)$   and $N_{\chi}^{+}(F)$  have two sides of degree $1$ each. By Theorem \ref{ore}, we see that $3 \Z_K = \pF_1  \pF_2 ( \pF_3  \pF_4)^2$, where   $f( \pF_k/3)=  1 $   for every $k = 1,2,3,4$. Thus for $p=3$, we have $P_3 = 4 > N_3(1)=3$. It follows by Lemma \ref{comindex}, that $3 \mid i(K)$. Consequently $K$ is not monogenic. 
\end{enumerate}
\end{proof}

\section{exemples}
Let $F(x) \in \Z[x]$ be a monic irreducible polynomial and $K$  a number field generated by a complex  root  of $F(x)$.  
\begin{enumerate}
	 \item  For $F(x) = x^{27} + 810 x + 2$, by Corollary \ref{3^r}, $3 \mid i(K)$ and so $K$ is not monogenic. 
	 \item {	For $F(x)= x^4+8x+8$, this polynomial has  $\Delta(F)= 5 \cdot 2^{12}$. Since  $2^2 \mid 8$, then by\cite[Theorem 1.1]{SK}, we conclude that  $2 \mid (\Z_K:\Z[\th])$, it  follows that  $\Z[\th] $  is not  the ring of integers of $K$. But, by Theorem \ref{mono},  the algebraic integer   $\eta = \frac{\th ^{3}}{4}$  generates a power integral basis of $\Z_K$, and so $K$ is monogenic.}
	 \item For $F(x)= x^{10}+161x+576$, by Theorem \ref{dn2}, $K$ is not monogenic. 

	\item 	For $F(x) = x^{18}+ax+b$, $a \equiv 9 \,\, \mbox{or}\,\, 18 \md{27}$ and $b\equiv 26 \md{27}$. By Theorem \ref{dn1}(2), $K$ is not monogenic.
	 \item For $F(x) = x^{7}+80x+54$, as $F(x)$ is $2$-Eisenstein polynomial, then it  is irreducible over $\Q$. Since $80 \equiv -1 \mod {27}$ and $54 \equiv 0 \mod{27}$. By Corollary \ref{corn12}(7), the the sepstic field $K$ is not monogenic. 
	 \item For $F(x) = x^6 + 270 x + 26$, then $F(x)$ is $2$-Eisenstein, then it is irreducible over $\Q$. By Theorem \ref{d61}(4), the sextic field $K$ is not monogenic.  
	 \item For $F(x)= x^{4 \cdot 5^r} + 775  x + 124$ with $r\ge 2$, as for $p = 31$ and $\ph = x$, $\npp{F} = S $ has a single side with slope $\l= \frac{-1}{4 \cdot 5^r}$. So $R_{\l}(F)(y)$ is irreducible over $\F_{ \ph} \simeq \F_{ 31}$ as it is of degree $1$, and so $F(x)$ is irreducible over $\Q$. As $x^4-1 \equiv (x-1)(x-2)(x-3)(x-4) \md 5$, then $N_5(1,4,4) = 4$.  Sine $\n_5(775) = 2$, $\n_5(124^4-1) \ge 3$ and $r \ge 2$.  By Theorem \ref{dn1}(2), $5  \mid i(K)$ and so $K$ is not monogenic. 
 \end{enumerate}

\end{document}